\documentclass[a4paper,11pt,reqno]{article}
\usepackage[utf8]{inputenc}

\usepackage[a4paper,top=2cm,bottom=2cm,left=2.2cm,right=2.2cm]{geometry}

\usepackage{titlesec}
\usepackage{lipsum}
\usepackage{enumerate}
\usepackage{enumitem}
\usepackage{mathrsfs}
\usepackage{amsmath,amssymb,amsthm}
\usepackage{mathtools}
\usepackage{graphicx, float, subfigure}
\usepackage{url}
\usepackage{tikz}
\usepackage{bbold}
\usepackage{mleftright}
\usepackage{color}
\usepackage{xcolor}
\usepackage{colortbl}
\usepackage{array}

\usepackage{ytableau}
\usepackage[all,cmtip]{xy}
\usepackage{multirow}
\usepackage{makecell}
\usepackage[colorinlistoftodos]{todonotes}
\usepackage[colorlinks=true, allcolors=red]{hyperref}
\usepackage{harpoon}
\usepackage{MnSymbol}
\usepackage{esvect}
\usepackage{diagbox}
\usepackage[title]{appendix}
\usepackage{authblk}

\theoremstyle{plain}
\newtheorem{theorem}{\scshape Theorem}[section]
\newtheorem{proposition}[theorem]{\scshape Proposition}
\newtheorem{lemma}[theorem]{\scshape Lemma}
\newtheorem{corollary}[theorem]{\scshape Corollary}

\newtheorem*{assumption*}{\scshape Assumption}

\newtheorem*{claim*}{Claim}

\theoremstyle{definition}

\newtheorem{remark}[theorem]{\scshape Remark}
\newtheorem{example}[theorem]{\scshape Example}

\newtheorem{problem}[theorem]{\scshape Problem}

\renewcommand{\det}{\mathsf{det}}

\numberwithin{equation}{section}

\titleformat{\section}{\centering\bfseries}{\thesection}{1em}{\MakeUppercase}
\titleformat{\subsection}{\bfseries}{\thesubsection}{1em}{}

\begin{document}
\title{Enumeration of pattern-avoiding $(0,1)$-matrices and their symmetry classes}

\author[1]{Sen-Peng Eu\thanks{\url{speu@math.ntnu.edu.tw}}}
\author[1]{Yi-Lin Lee\thanks{\url{yillee@ntnu.edu.tw}, the corresponding author.}} 
\affil[1]{Department of Mathematics, National Taiwan Normal University, Taipei, Taiwan}

\date{}
\maketitle

\begin{abstract}
    Recently, Brualdi and Cao studied $I_k$-avoiding $(0,1)$-matrices by decomposing them into zigzag paths and proved that the maximum number of $1$'s in such a matrix is given by an exact formula. We further study the structure of maximal $I_k$-avoiding $(0,1)$-matrices (IAMs) by interpreting them as families of non-intersecting lattice paths on the square lattice. Using this perspective, we establish a bijection showing that IAMs are equinumerous with plane partitions of a certain size. Moreover, we classify all ten symmetry classes of IAMs under the action of the dihedral group of order $8$ and show that the enumeration formulas for these classes are given by simple product formulas. Extending this approach to skew shapes, we derive a conceptual formula for enumerating maximal $I_k$-avoiding $(0,1)$-fillings of skew shapes.
\end{abstract}
\begin{small}

\vspace{-1mm}
\end{small}

\section{Introduction and statement of results}\label{sec.intro}

A \emph{permutation} $p$ is an array $p_1p_2\cdots p_n$ on the first $n$ positive integers in some order, so that each integer appears exactly once. We say a permutation $p$ \emph{contains} a pattern $q =q_1q_2\cdots q_k$ if there exist indices $i_1<i_2<\cdots <i_k$ so that $p_{i_s} < p_{i_t}$ if and only if $q_s < q_t$. On the other hand, we say $p$ \emph{avoids} $q$ or $p$ is \emph{$q$-avoiding} if $p$ does not contain $q$. Patterns of permutations, especially pattern-avoiding permutations, have been investigated extensively. We refer the reader to the book written by Kitaev \cite{Kitaev11} and B\'ona \cite[Chapters 4 and 5]{Bona} for a comprehensive review.


%
%

The notion of pattern avoidance has been extended from permutations to $(0,1)$-matrices. Let $M$ be an $m \times n$ $(0,1)$-matrix and $Q$ an $s \times t$ $(0,1)$-matrix. We say $M$ \emph{contains} $Q$ if there is a selection of rows $(r_1,\dots,r_s)$ and columns $(c_1,\dots,c_t)$ of $M$ such that the $(r_i,c_j)$-entry of $M$ equals $1$ if the $(i,j)$-entry of $Q$ equals $1$. We say $M$ is \emph{$Q$-avoiding} if there is no such selection of rows and columns. 

There are two main directions of research on pattern-avoiding matrices. 

The first concerns an extremal problem: given a $Q$-avoiding $(0,1)$-matrix $M$, what is the maximum number of $1$'s that $M$ can contain? 
To the best of our knowledge, most results give a bound of the maximum number of $1$'s. For instance, Marcus and Tardos \cite{MT04} proved that when $Q$ is a $k \times k$ permutation matrix and $M$ is a square matrix of size $n$, then the number of $1$'s in a $Q$-avoiding $(0,1)$-matrix $M$ is bounded above by $2k^4 \binom{k^2}{k}n$; see also \cite{Keszegh05} and \cite{Cibulka13} for more results. Recently, Brualdi and Cao \cite{BC21} showed that when $Q=I_k$ (i.e., the $k \times k$ identity matrix), then the maximum number of $1$'s that $M$ can contain is given by an exact number (see Section \ref{sec.introIAM}).

The second direction investigates the structural relations between $Q_1$-avoiding $(0,1)$-matrices and $Q_2$-avoiding $(0,1)$-matrices for two distinct matrices $Q_1$ and $Q_2$. Inspired by the seminal work of Chen et al. \cite{CDDSY07} on $k$-noncrossing and $k$-nonnesting matchings and partitions, Krattenthaler \cite{Kra06} and de Mier \cite{deMier07} generalized to pattern-avoiding fillings of Young diagrams with nonnegative integers. Their work focuses primarily on proving equalities of the number of $I_k$-avoiding fillings of a Young diagram under various restrictions and the number of $J_k$-avoiding (i.e., avoid the $k \times k$ anti-identity matrix) fillings of the same shape under the same restriction. Such restrictions may include prescribed row sums and column sums, the number of $1$'s, the length of ``NE-chains'' and ``SE-chains''. Jonsson \cite{Jonsson05} further generalized this framework to pattern-avoiding fillings of stack polyominoes from a topological and geometric perspective, interpreting such fillings as simplicial complexes.

In this paper, we study the structure of ``maximal'' $I_k$-avoiding $(0,1)$-matrices (as defined in Section \ref{sec.introIAM}) and show that they are equinumerous with plane partitions of a certain size (Theorem \ref{thm1}). Moreover, we provide a complete classification of symmetry classes of these matrices (as shown in Section \ref{sec.introsym}) and show that the counting formulas for these symmetry classes are given by simple product formulas (Theorem \ref{thm2}). In Section \ref{sec.introskew}, we extend our study to``maximal'' $I_k$-avoiding $(0,1)$-fillings of skew shapes and provide a conceptual formula for enumerating these fillings (Theorem \ref{thm4}).

\subsection{Maximal $I_k$-avoiding $(0,1)$-matrices}\label{sec.introIAM}


Throughout the paper, we assume that the matrix has size $m \times n$ and is $I_k$-avoiding, where $m,n$, and $k$ are positive integers with $2 \leq k \leq \min\{m,n\}$. Sometimes, we may without loss of generality assume $m \leq n$.

We consider \emph{maximal} $I_k$-avoiding $(0,1)$-matrices $M$ (IAMs), in the sense that $M$ contains maximally many $1$'s; in other words, if we replace any $0$ entry by $1$ in $M$, then $M$ contains $I_k$. This terminology is also called \emph{$12\cdots k$-avoiding} by Brualdi and Cao \cite{BC21}. Let $\mathcal{M}_{m,n;k}$ denote the set of maximal $I_k$-avoiding $(0,1)$-matrices $M$ of size $m \times n$. 



Brualdi and Cao \cite[Lemma 6 and Theorem 7]{BC21} proved that
\begin{itemize}
    \item the maximum number of $1$'s of any matrix $M$ in $\mathcal{M}_{m,n;k}$ is $(k-1)(m+n-k+1)$, and
    \item every such matrix $M$ can be decomposed into $(k-1)$ \emph{zigzag paths}\footnote{They are called \emph{R-L zigzag paths} in \cite{BC21}.} of lengths $m+n-(2i-1)$, $i=1,2,\dots,k-1$, where a zigzag path of length $d$ is the set of $d$ connected entries of $M$ consisting of $1$'s, going in southwestern-northeastern direction, and they do not form a $2 \times 2$ block in $M$. 
\end{itemize}
Our interest is to enumerate maximal $I_k$-avoiding $(0,1)$-matrices $M$ of size $m \times n$. 

The main contribution of this paper gives a bijection between $\mathcal{M}_{m,n;k}$ and the set of certain plane partitions. A \emph{plane partition} of size $a \times b \times c$ is an $a \times b$ array $\pi=(\pi_{i,j})$ of nonnegative integers such that $\pi_{i,j} \leq c$ and these integers are weakly decreasing along rows and columns, that is, $\pi_{i,j} \geq \pi_{i+1,j}$ and $\pi_{i,j} \geq \pi_{i,j+1}$. Let $\mathcal{PP}(a,b,c)$ denote the set of all plane partitions of size $a \times b \times c$. The celebrated result due to MacMahon \cite{MacM} states that the number of plane partitions of size $a \times b \times c$ is given by the simple product formula $H(a,b,c)$.
\begin{equation}\label{eq.Hproduct}
    |\mathcal{PP}(a,b,c)| = H(a,b,c) := \prod_{i=1}^{a} \prod_{j=1}^{b} \prod_{\ell=1}^{c} \frac{i+j+\ell-1}{i+j+\ell-2}.
\end{equation}

We state the main result in the following theorem.
\begin{theorem}\label{thm1}
    Let $m,n$, and $k$ be positive integers with $2 \leq k \leq \min \{m,n\}$. The set of maximal $I_k$-avoiding $(0,1)$-matrices $M$ of size $m \times n$ is in bijection with the set of plane partitions of size $(m-k+1) \times (n-k+1) \times (k-1)$. Therefore, 
    \begin{equation}\label{eq.thm1}
        |\mathcal{M}_{m,n;k}| = H(m-k+1,n-k+1,k-1).
    \end{equation}
\end{theorem}

\subsection{Symmetry classes of IAMs}\label{sec.introsym}

The first application of Theorem \ref{thm1} allows us to classify and enumerate all the symmetry classes of IAMs. In general, assuming a finite group $G$ acts on a set of combinatorial objects $X$ is well-defined, and $H$ is a subgroup of $G$. A \emph{symmetry class} of $X$ is a collection of $H$-invariant objects of $X$. Enumerating each symmetry class of $X$ is challenging, since the structure of each class depends on the choice of a subgroup $H$, which typically requires different enumeration techniques. 

During the 1980s, the study of symmetry classes of combinatorial objects emerged as a central topic in enumerative combinatorics. In this paper, we will discuss in detail (in Section \ref{sec.sym}) $10$ symmetry classes of plane partitions identified by Stanley \cite{Stan86pp}; see also the survey paper by Krattenthaler \cite{Kra15S} and references therein. Symmetry classes of other combinatorial objects have been studied extensively since then, such as alternating sign matrices (\cite{Stan86}, \cite{Rob91,Rob00}, \cite[Section 2]{BFK23}), domino tilings of the Aztec diamond (\cite{Yang91}, \cite[Section 7]{Ciucu97}), and spanning trees of the Aztec diamond graphs (\cite{Ciucu08}). 

The dihedral group $D_8$ of order $8$ acts on the set $\mathcal{M}_{m,n;k}$ of all IAMs as symmetries of the square. This results in $10$ symmetry classes of IAMs shown in the list below. Note that since our matrices are $I_k$-avoiding, diagonally symmetric IAMs and anti-diagonally symmetric IAMs belong to different symmetry classes, as well as vertically symmetric IAMs and horizontally symmetric IAMs. Although Theorem \ref{thm1} shows that $\mathcal{M}_{m,n;k}$ is equinumerous with plane partitions, the structure of these two symmetry classes is different.
\begin{itemize}
    \item[(U)] \emph{Unrestricted} IAMs.
    
    \item[(DS)] \emph{Diagonally symmetric} IAMs that are symmetric in the main diagonal.
    
    \item[(AS)] \emph{Anti-diagonally symmetric} IAMs that are symmetric in the anti-diagonal.
    
    \item[(DAS)] \emph{Diagonally and anti-diagonally symmetric} IAMs that are symmetric in both diagonals.
    
    \item[(VS)] \emph{Vertically symmetric} IAMs that are invariant under flips around the vertical axis.
    
    \item[(HS)] \emph{Horizontally symmetric} IAMs that are invariant under flips around the horizontal axis. 
    
    \item[(VHS)] \emph{Vertically and horizontally symmetric} IAMs that are invariant under flips around the vertical axis and the horizontal axis.

    \item[(QTS)] \emph{Quarter-turn symmetric} IAMs that are invariant under a $90^{\circ}$ rotation.

    \item[(HTS)] \emph{Half-turn symmetric} IAMs that are invariant under a $180^{\circ}$ rotation.

    \item[(TS)] \emph{Totally symmetric} IAMs that are invariant under $D_8$.
\end{itemize}
For each symmetry class $* = U, DS, AS,\dots$, let $\mathcal{M}^{*}_{m,n;k}$ denote the set of maximal $I_k$-avoiding matrices of size $m \times n$ under the symmetry $*$. 

Five of the symmetry classes of IAMs are closely related to five symmetry classes of plane partitions (see Section \ref{sec.sym} for details). Consequently, the enumeration formulas for these five classes can be expressed as simple product formulas. The remaining five symmetry classes of IAMs are trivial. The result is stated in the following theorem.
\begin{theorem}\label{thm2}
    Let $m,n$, and $k$ be positive integers with $2 \leq k \leq \min \{m,n\}$.
    \begin{enumerate}
        \item The number of unrestricted $m \times n$ IAMs is given by
        \begin{equation}\label{eq.symU}
            |\mathcal{M}_{m,n;k}^{U}| = H(m-k+1,n-k+1,k-1).
        \end{equation}

        \item The number of diagonally symmetric $n \times n$ IAMs is given by
        \begin{equation}\label{eq.symDS}
            |\mathcal{M}_{n,n;k}^{DS}| = \prod_{1 \leq i \leq j \leq n-k+1} \frac{k+i+j-2}{i+j-1}.
        \end{equation}

        \item The number of anti-diagonally symmetric $n \times n$ IAMs is given by
        \begin{equation}\label{eq.symAS}
            |\mathcal{M}_{n,n;k}^{AS}| = \begin{cases} \displaystyle 
                \binom{n-\frac{k+1}{2}}{n-k}\prod_{1 \leq i \leq j \leq n-k-1} \frac{k+i+j}{i+j+1}, &\text{if $k$ is odd,}\\
                0, & \text{if $k$ is even}.
            \end{cases}
        \end{equation}

        \item The number of diagonally and anti-diagonally symmetric $n \times n$ IAMs is given by
        \begin{equation}\label{eq.symDAS}
            |\mathcal{M}_{n,n;k}^{DAS}| = \begin{cases}
            H(\frac{n-k+2}{2}, \frac{n-k}{2}, \frac{k-1}{2})    , &\text{ if $k,n$ are odd,}\\
            H(\frac{n-k+1}{2}, \frac{n-k+1}{2}, \frac{k-1}{2})    , &\text{ if $k$ is odd and $n$ is even.}\\
            0, &\text{ if $k$ is even.}
            \end{cases} 
        \end{equation}

        \item The number of half-turn symmetric $m \times n$ IAMs is given by the following formulas.
        \begin{itemize}
            \item If $k$ is even, then 
                \begin{equation}\label{eq.symHTSeven}
                |\mathcal{M}_{m,n;k}^{HTS}| = \begin{cases}
                H(\frac{m-k+1}{2}, \frac{n-k+1}{2}, \frac{k}{2})H(\frac{m-k+1}{2}, \frac{n-k+1}{2}, \frac{k-2}{2}) , &\text{ for $m,n$ odd,} \\
                H(\frac{m-k+1}{2}, \frac{n-k+2}{2}, \frac{k-2}{2})H(\frac{m-k+1}{2}, \frac{n-k}{2}, \frac{k}{2})     , &\text{ for $m$ odd and $n$ even,} \\
                0 , &\text{ for $m,n$ even.}
                \end{cases} 
            \end{equation}
            \item If $k$ is odd, then
            \begin{equation}\label{eq.symHTSodd}
                |\mathcal{M}_{m,n;k}^{HTS}| = \begin{cases}
                H(\frac{m-k+2}{2}, \frac{n-k}{2}, \frac{k-1}{2})H(\frac{m-k}{2}, \frac{n-k+2}{2}, \frac{k-1}{2}) , &\text{ for $m,n$ odd,} \\
                H(\frac{m-k+1}{2}, \frac{n-k+2}{2}, \frac{k-1}{2})H(\frac{m-k+1}{2}, \frac{n-k}{2}, \frac{k-1}{2})     , &\text{ for $m$ odd and $n$ even,} \\
               H(\frac{m-k+1}{2}, \frac{n-k+1}{2}, \frac{k-1}{2})^2  , &\text{ for $m,n$ even.}
                \end{cases} 
            \end{equation}
        \end{itemize}

        \item The number of the remaining five symmetry classes is given by 
        \begin{equation}
            |\mathcal{M}_{m,n;k}^{VS}|=|\mathcal{M}_{m,n;k}^{HS}| = |\mathcal{M}_{m,n;k}^{VHS}| = |\mathcal{M}_{n,n;k}^{QTS}| = |\mathcal{M}_{n,n;k}^{TS}| = \begin{cases}
                1, & \text{if $k$ is odd,}\\
                0, & \text{if $k$ is even.}
            \end{cases}
        \end{equation}  
    \end{enumerate}
\end{theorem}

\subsection{Fillings of the Young diagram of skew shapes}\label{sec.introskew}

The second application of Theorem \ref{thm1} provides a new viewpoint from maximal $I_k$-avoiding $(0,1)$-matrices to maximal $I_k$-avoiding $(0,1)$-fillings of regions on the square lattice other than rectangles. In general, let $F$ be a filling of each box of a region $R$ on the square lattice with $0$ or $1$. We say $F$ \emph{contains} $I_k$ if there is a selection of rows $(r_1,\dots,r_k)$ and columns $(c_1,\dots,c_k)$ such that the box $(r_i,c_j) = \delta_{i,j}$, the Kronecker delta function, and the boxes $\{(r_i,c_j)|i,j=1,2,\dots,k\}$ are contained in $R$.  We say $F$ is \emph{$I_k$-avoiding} if there is no such selection of rows and columns. Similarly, we consider a maximal $I_k$-avoiding $(0,1)$-filling $F$ of $R$, in the sense that, $F$ contains maximally many $1$'s. Let $\mathcal{F}_{R;k}$ be the set of maximal $I_k$-avoiding $(0,1)$-fillings of $R$.

We first study the region $R=\bar{R}_{m,n;t}$ obtained from an $m \times n$ rectangle (assuming $m \leq n$) by removing a staircase of size $t$ on the lower right corner. We prove that when $t=m-k$ or $t=m-k+1$, the number of maximal $I_k$-avoiding $(0,1)$-fillings of $\bar{R}_{m,n;t}$ is given by a nice product formula. Our data shows that for other values of $t$, the number contains some large prime factors, so having a product formula seems unlikely. The specialty of these two values of $t$ can be regarded as removing a ``maximal'' staircase on the lower right corner from the rectangle. 

\begin{theorem}\label{thm3}
    Let $m,n$, and $k$ be positive integers with $2 \leq k \leq m \leq n$. Assume $t=m-k$ or $m-k+1$. Then the number of maximal $I_k$-avoiding $(0,1)$-fillings of $\bar{R}_{m,n;t}$ is given by
    \begin{equation}\label{eq.maxstaircaseprod}
        |\mathcal{F}_{\bar{R}_{m,n;t};k}| = \prod_{i=1}^{k-1}\frac{(m+n-2k+2i+\delta_{t,m-k})!}{(m-i)!(n+i-1+\delta_{t,m-k})!} \cdot \prod_{i=1}^{k-2}i! \cdot \prod_{1 \leq i \leq j \leq k-1}(n-m+i+j-1+\delta_{t,m-k}),
    \end{equation}
    where $\delta_{i,j}$ is the Kronecker delta function.
\end{theorem}

Our method allows us to enumerate maximal $I_k$-avoiding $(0,1)$-fillings of $R$ where $R$ is the Young diagram of a skew shape $\lambda/\mu$, which is introduced as follows. A \emph{partition} $\lambda=(\lambda_1,\lambda_2,\dotsc,\lambda_{\ell})$ of $n$ is a weakly decreasing sequence of nonnegative integers whose sum is $n$. Denote by $\mathsf{n}(\lambda)$ the number of parts $\ell$ of the partition $\lambda$. By convention, when $\ell=0$, $\lambda = \emptyset$ is the empty partition. 

The \emph{Young diagram} of shape $\lambda$ is a collection of boxes arranged in left-justified rows, with $\lambda_i$ boxes in the $i$th row. For another partition $\mu=(\mu_1,\dots,\mu_{\ell})$, we write $\mu \subset \lambda$ if $\mu$ is contained in $\lambda$ as Young diagrams, that is, $\mu_i \leq \lambda_i$ for all $i$. In this case, the Young diagram of the skew shape $\lambda/\mu$ is obtained by removing from the diagram of $\lambda$ all boxes that also belong to the diagram of $\mu$. 

Now, let $\lambda/\mu$ be a skew shape, where $\lambda_1=n$ and $\mathsf{n}(\lambda)=m$. The Young diagram of $\lambda/\mu$ is inscribed in the rectangle of size $m \times n$. When considering maximal $I_k$-avoiding $(0,1)$-fillings of $\lambda/\mu$, we require that the staircase of size $k$ can be placed in the lower left corner and the upper right corner of $\lambda/\mu$. This requirement can be expressed as follows.
\begin{equation}\label{eq.skewcondition}
    \begin{cases}
        \lambda_m \geq k, \\
        \mu \text{ has at least $k$ zero parts,}\\
        \lambda_1 - \mu_1 \geq k, \text{ and}\\ \lambda_1=\lambda_2=\cdots=\lambda_{\alpha} = n \text{ for $\alpha \geq k$.}
    \end{cases}
\end{equation}

Following the idea of viewing maximal $I_k$-avoiding fillings $F$ as non-intersecting lattice paths, we provide a conceptual formula for the number of fillings $F$ of a skew shape $\lambda/\mu$, which is expressed as a determinant of a matrix, and the entries of the matrix are relatively complicated. Evaluating such a determinant is challenging; this can be computed for each skew shape $\lambda/\mu$ case by case. Currently, we have found that the only shape that has a nice product formula is presented in Theorem \ref{thm3}.

Before stating the conceptual formula, we need some notations and definitions. First, we take the \emph{planar dual graph} of the Young diagram of $\lambda/\mu$, that is, identify each box in the Young diagram as a vertex, and if two boxes are adjacent, then connect the two corresponding vertices with an edge. Since $k \geq 2$, under the requirement \eqref{eq.skewcondition}, the resulting planar dual graph is again the Young diagram of a partition denoted by $\bar{\lambda}/\bar{\mu}$. 

Second, let $a,b$ be two nonnegative integers. We define the function on partitions 
\begin{equation}\label{eq.gamma}
    \gamma_{a,b}(\lambda) = (\lambda_{a+1},\lambda_{a+2},\dots,\lambda_{m-b-1},\lambda_{m-b})
\end{equation}
by deleting the first $a$ parts and the last $b$ parts of $\lambda$. Clearly, $\gamma_{0,0}(\lambda)=\lambda$. For convenience, we also define the function on two partitions $\mu \subset \lambda$ (some parts of $\mu$ may be $0$)
\begin{equation}\label{eq.f}
    f(\lambda,\mu) = \det \left[ \binom{\lambda_j-\mu_i+1}{j-i+1} \right]_{i,j=1}^{\mathsf{n}(\lambda)}.
\end{equation}

The conceptual formula is presented in the following Theorem.
\begin{theorem}\label{thm4}
    Assume $m,n$, and $k$ are positive integers with $2 \leq k \leq \min\{m,n\}$. Let $\lambda=(\lambda_1,\lambda_2,\dotsc,\lambda_{m})$ and $\mu=(\mu_1,\mu_2,\dotsc,\mu_{m})$ be two partitions such that $\mu \subset \lambda$, $\lambda_1=n$, and they satisfy \eqref{eq.skewcondition}. We write $\bar{\lambda}/\bar{\mu}$ for the skew partition whose Young diagram is the planar dual graph of $\lambda/\mu$. Then the number of maximal $I_k$-avoiding $(0,1)$-fillings of $\lambda/\mu$ is given by
    \begin{equation}\label{eq.thm4}
        |\mathcal{F}_{\lambda/\mu; k}| = \det [f(\gamma_{k-1-j,i-1}(\bar{\lambda}),\gamma_{k-1-j,i-1}(\bar{\mu}))]_{i,j=1}^{k-1}.
    \end{equation}
\end{theorem}

The rest of this paper is organized as follows. In Section \ref{sec.main}, we prove Theorem \ref{thm1} by establishing a bijection between plane partitions and IAMs. We also consider a weighted enumeration of IAMs and show that it is given by a triple product formula. In Section \ref{sec.sym}, we investigate the symmetry classes of IAMs and prove Theorem \ref{thm2}. In Section \ref{sec.skew}, we study maximal $I_k$-avoiding $(0,1)$-fillings of $\lambda/\mu$ and prove Theorems \ref{thm3} and \ref{thm4}. Finally, some research directions are proposed in Section \ref{sec.remarks}.

\section{The bijection and the weighted enumeration}\label{sec.main}

We introduce the correspondence between plane partitions and non-intersecting lattice paths in Section \ref{sec.mainPP}. In Section \ref{sec.mainIAM}, we analyze the structure of maximal $I_k$-avoiding $(0,1)$-matrices (IAMs) and prove Theorem \ref{thm1}. In Section \ref{sec.mainweight}, we provide a triple product formula for the weighted enumeration of IAMs.

\subsection{Plane partitions and non-intersecting lattice paths}\label{sec.mainPP}

Plane partitions $\pi=(\pi_{i,j})$ are originally defined as fillings of a rectangular array with positive integers that are weakly decreasing along rows and columns. There is a natural way to visualize a plane partition of size $a \times b \times c$ as a pile of unit cubes inside the $a \times b \times c$ box, where a stack of $\pi_{i,j}$ unit cubes is placed at the position $(i,j)$ in the box. See Figures \ref{fig.pparray} and \ref{fig.ppbox} for an example of a plane partition in $\mathcal{PP}(5,3,4)$. 

One can encode a plane partition $\pi=(\pi_{i,j}) \in \mathcal{PP}(a,b,c)$ by the sequence of partitions $(\lambda^{s}(\pi))_{s=1}^{c}$, where $\lambda^{s}(\pi)$ is the partition corresponding to the shape of $\{\pi_{i,j} | \pi_{i,j} \geq s\}$ contained in $\pi$, for $s=1,2,\dots,c$. In other words, viewing $\pi$ as a pile of unit cubes, $\lambda^{s}(\pi)$ is the shape of the $s$th layer of this pile. In our running example (Figure \ref{fig.pp}), the sequence contains four terms: $\lambda^{1} = (3,3,3,2,1), \lambda^{2}=(3,3,2),\lambda^{3}=(2,2,1)$, and $\lambda^{4} = \emptyset$.

Moreover, one can interpret a pile of unit cubes as non-intersecting lattice paths, where each path traverses through the ``surface'' of each layer of the pile. We describe it in detail as follows (see Figure \ref{fig.pppath}). First, we view a pile of unit cubes as a graph drawing on the plane and mark the midpoint of each vertical edge of this graph. The paths start at the midpoints on the left side of the graph, and the steps of each path always connect the midpoints of opposite sides of rhombi until the paths reach the right side of the graph. We then obtain in this way the collection of paths connecting the midpoints on the left side with the midpoints on the right side. Clearly, these paths are \emph{non-intersecting}, meaning that any two paths do not pass through the same point. 

By slightly deforming the paths, we may place them on the square lattice (see Figure \ref{fig.pplattice}). Let $U=\{u_i=(c-i,i-1)|i=1,\dots,c\}$ and $V=\{v_j = (b+c-j,a+j-1)|j=1,\dots,c\}$ be two sets of lattice points. We write $\mathcal{P}(u_i,v_j)$ for the set of paths going from $u_i$ to $v_j$ and $\mathcal{P}(U,V)$ for the set of non-intersecting lattice paths $(p_1,\dots,p_{c})$, where $p_i \in \mathcal{P}(u_i,v_i)$ for $i=1,\dots,c$. A path that starts and ends at the same point is considered as the path of length zero. From the above discussion, we have obtained a bijection between the set of plane partitions $\mathcal{PP}(a,b,c)$ and families of non-intersecting lattice paths $\mathcal{P}(U,V)$.
\begin{figure}[hbt!]
    \centering
    \subfigure[]{\label{fig.pparray}\includegraphics[height=0.18\textwidth]{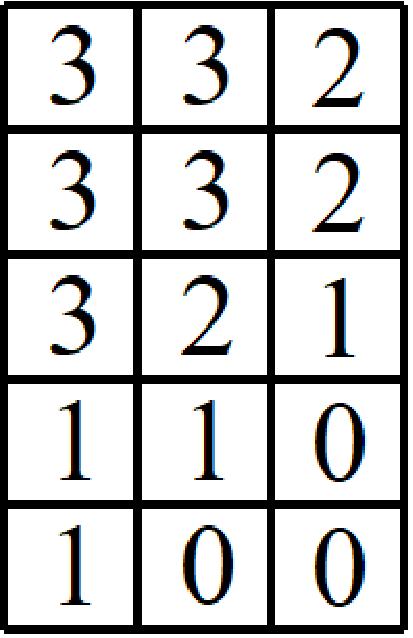}}
    \hspace{4mm}
    \subfigure[]
    {\label{fig.ppbox}\includegraphics[height=0.27\textwidth]{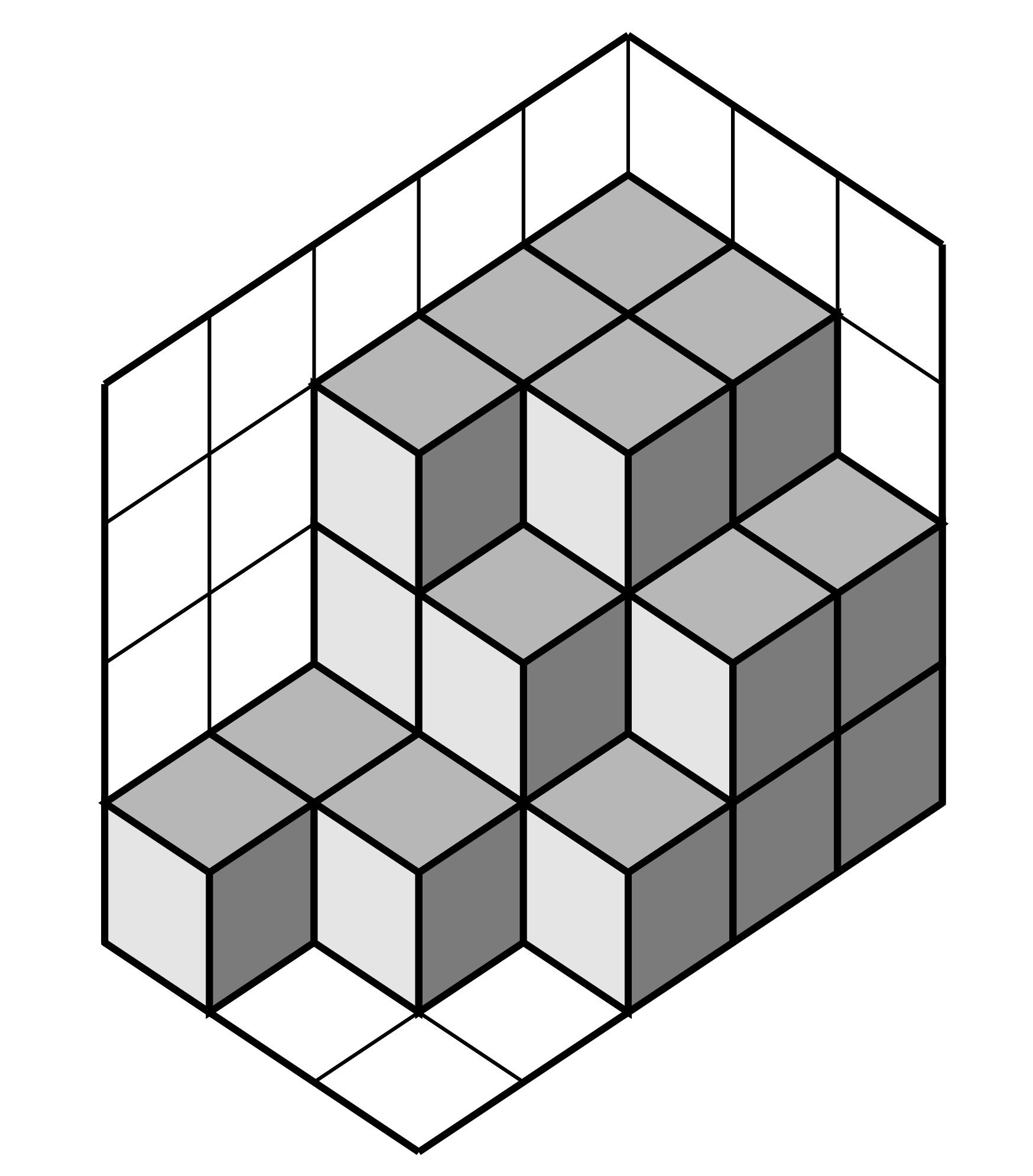}}
    \hspace{2mm}
    \subfigure[]
    {\label{fig.pppath}\includegraphics[height=0.27\textwidth]{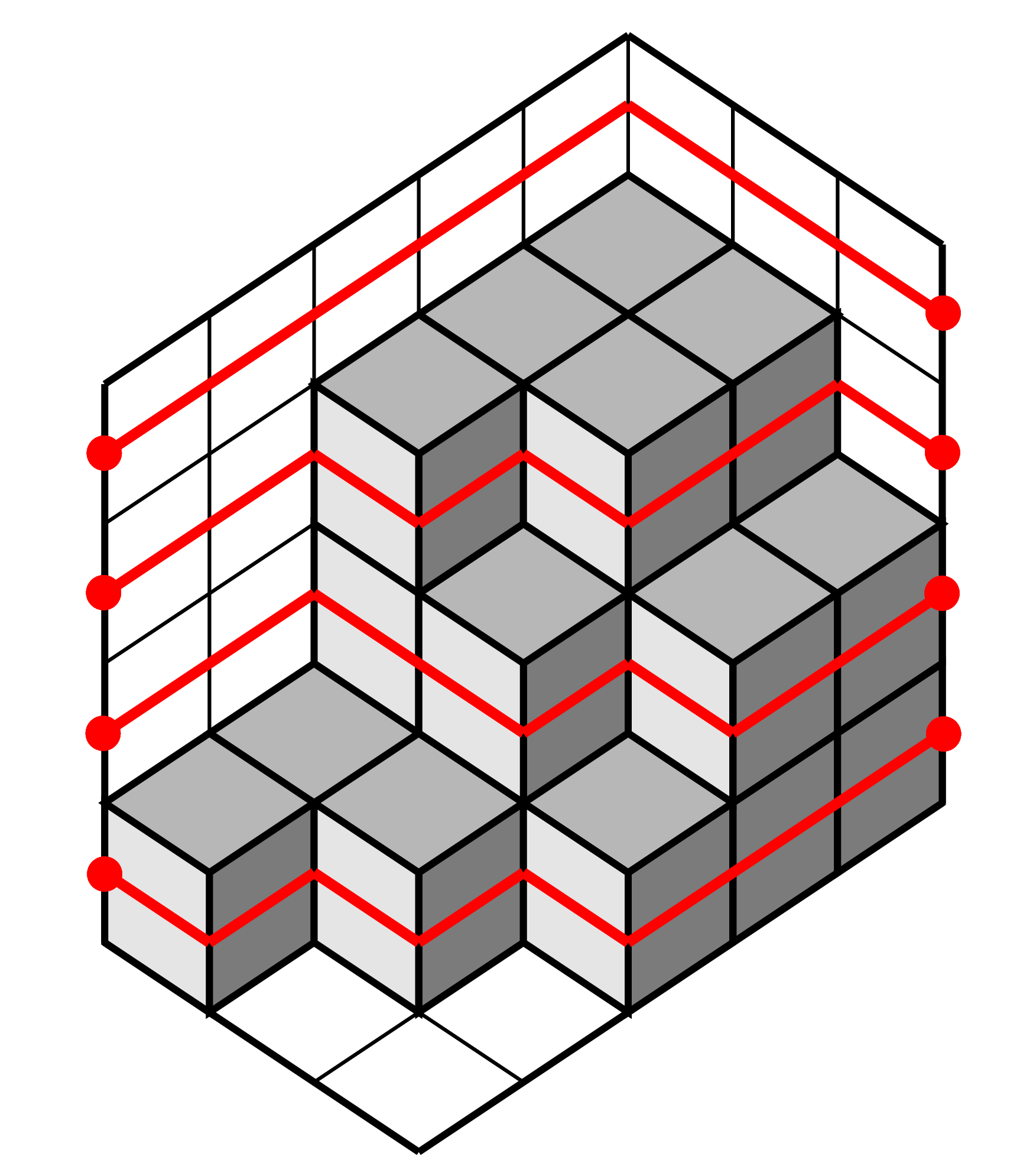}}
    \hspace{2mm}
    \subfigure[]{\label{fig.pplattice}\includegraphics[height=0.27\textwidth]{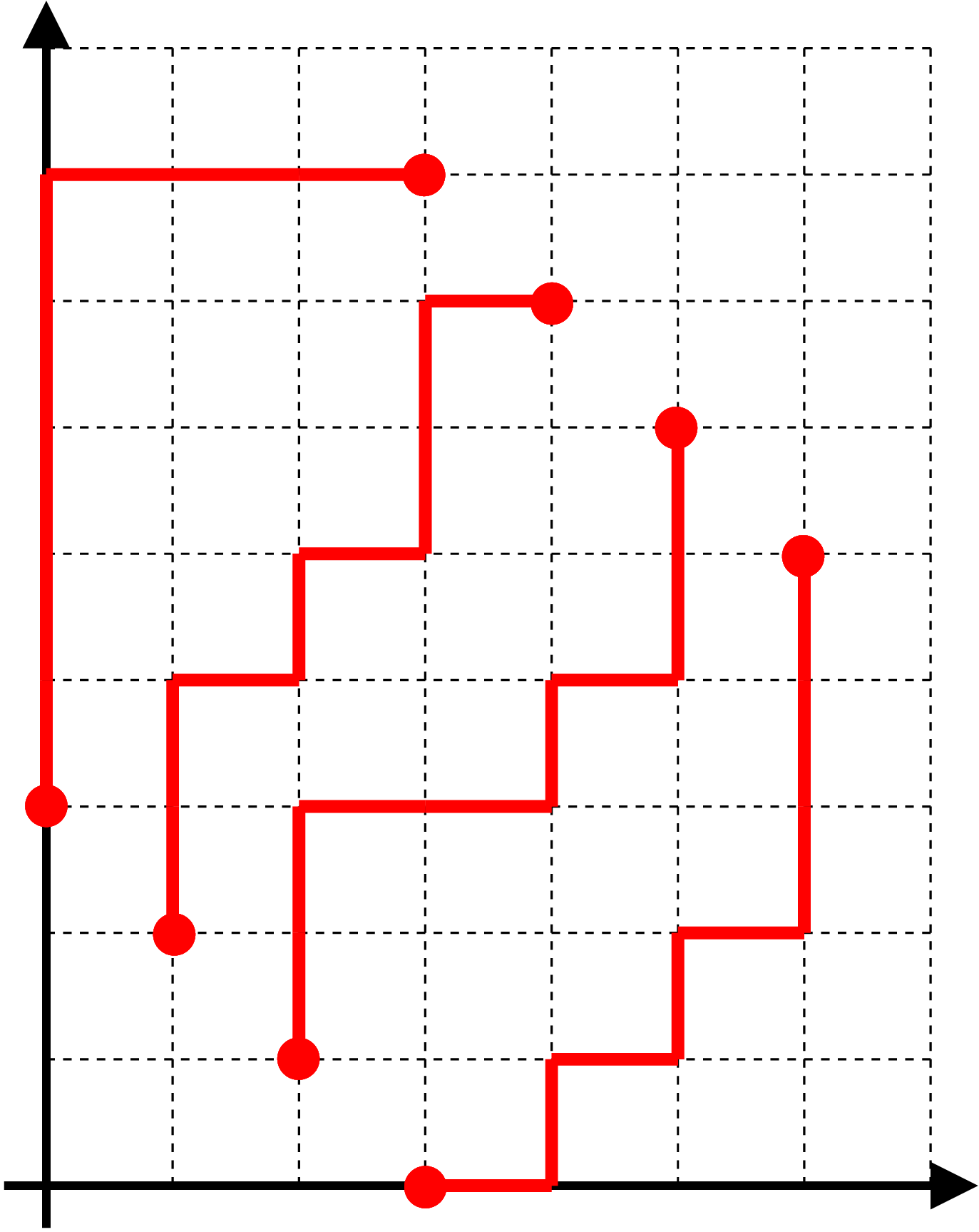}}
    \caption{(a) A plane partition $\pi \in \mathcal{PP}(5,3,4)$. (b) Viewing $\pi$ as a pile of unit cubes in the $5 \times 3 \times 4$ box. (c) The non-intersecting paths encoding of $\pi$. (d) The corresponding non-intersecting lattice paths of $\pi$ which are obtained from deforming the paths in Figure \ref{fig.pppath}.}
    \label{fig.pp}
\end{figure}

\subsection{Maximal $I_k$-avoiding $(0,1)$-matrices and non-intersecting lattice paths}\label{sec.mainIAM}

Let $M$ be a maximal $I_k$-avoiding $(0,1)$-matrix of size $m \times n$. We observe that the entries $\{(i,j)|1 \leq j \leq k-1, m-k+j+1 \leq i \leq m\}$ and $\{(i,j)|1 \leq i \leq k-1, n-k+i+1 \leq j \leq n\}$ of $M$ must be filled with $1$'s. Otherwise, if there is a $0$ in these entries, one can replace it with $1$, and $M$ is still $I_k$-avoiding; this contradicts the maximality of the number of $1$'s in $M$. Note that these entries form two staircases located at the lower left and upper right corners of $M$, respectively. 

Let $\{\ell_i=(m-i+1,k-i)|i=1,\dots,k-1\}$ (resp., $\{r_i=(k-i,n-i+1)|i=1,\dots,k-1\}$) be the ``diagonal'' entries in the staircase located at the lower left (resp., upper right) corner of $M$. See Figure \ref{fig.matrixstaircases} for an illustration. As presented in Section \ref{sec.introIAM}, Brualdi and Cao showed that the matrix $M$ can be decomposed into $k-1$ zigzag paths $p_1,p_2,\dots p_{k-1}$ (assuming going in the northeastern direction in $M$) of lengths $m+n-1,m+n-3,\dots,m+n-(2k-3)$, respectively.
    \begin{figure}[hbt!]
        \centering
        $\left[
            \begin{array}{c|c|c|c|c|c|c|c}
              &  &  & r_{k-1} & * & \cdots & * & *  \\
              \hline
              &  &  &  & r_{k-2} & \ddots & * & * \\
              \hline
              &  &  &  &  & \ddots & * & * \\
              \hline
              &  &  &  &  &  & r_2 & * \\
              \hline
              &  &  &  &  &  &  & r_1 \\
              \hline
             \ell_{k-1} &  &  &  &  &  &  &  \\
              \hline
             * & \ell_{k-2} &  &  &  &  &  & \\
              \hline
             \vdots & \ddots & \ddots &  &  &  &  & \\
              \hline
             * & * & * & \ell_2 &  &  &  &  \\
              \hline
             * & * & * & * & \ell_1 &  &  & 
            \end{array}
        \right]$
        \caption{An illustration of two staircases of size $k-1$ in a matrix. The ``diagonal'' entries of staircases are labeled $\ell_i$'s and $r_i$'s, and the entries marked $*$ must be filled with $1$'s.}
        \label{fig.matrixstaircases}
    \end{figure}

In the following lemma, the structure of these zigzag paths obtained from the decomposition of Brualdi and Cao is presented.
\begin{lemma}\label{lem.zigzags}
    For each $i=1,2,\dots,k-1$, the zigzag path $p_i$ must connect $\ell_{j}$ with $r_{j}$ for some $j$, and contains $k-1-i$ entries in both lower left and upper right staircases, respectively.
\end{lemma}
\begin{proof}
    Since these $k-1$ zigzag paths are non-intersecting, each zigzag path must pass through $\ell_i$ and $r_i$ for some $i=1,\dots,k-1$. For each $i$, a zigzag path connecting $\ell_i$ and $r_i$ has length $m+n-(2k-3)$, which is the minimum length of the zigzag paths in the decomposition of Brualdi and Cao.
    
    To fulfill the length restrictions of these zigzag paths, the longest zigzag path $p_1$ connects $\ell_{i_1}$ and $r_{i_1}$ for some $i_1$, and also passes through $k-2$ entries in both lower left and upper right staircases, respectively. In other words, $p_1$ starts at the $(m,1)$-entry and ends at the $(1,n)$-entry of $M$ and its length is given by $m+n-1$. Now, the second longest one connects $\ell_{i_2}$ and $r_{i_2}$ for some $i_2 \neq i_1$, and passes through $k-3$ entries\footnote{Since $k-2$ entries have been taken for $p_1$, we cannot take, say $k-2$ entries in one staircase and $k-4$ entries in the other staircase for $p_2$.} in both lower left and upper right staircases, respectively. 
    
    Continuing this argument, we arrive at the zigzag path $p_{k-1}$, which starts at $\ell_{i_{k-1}}$ and ends at $r_{i_{k-1}}$. Therefore, $p_i$ connects $\ell_{j}$ with $r_{j}$ for some $j$, and contains $k-1-i$ entries in both lower left and upper right staircases, respectively.
\end{proof}

However, the decomposition of $M$ into $k-1$ zigzag paths is not unique. This is due to the selection of $k-1-i$ entries of $p_i$ in both staircases. In Figure \ref{fig.example}, we present two different decompositions of a maximal $I_5$-avoiding $(0,1)$-matrix of size $9 \times 7$ into $4$ zigzag paths. Their lengths are $15$ (gray), $13$ (light gray), $11$ (orange), and $9$ (pink).
\begin{figure}[hbt!]
    \centering
    \subfigure{
    $\left[
    \begin{array}{c|c|c|c|c|c|c}
      \cellcolor{gray}{1} & \cellcolor{gray}{1} & \cellcolor{gray}{1} & \cellcolor{gray}{1} & \cellcolor{gray}{1} & \cellcolor{gray}{1} & \cellcolor{gray}{1} \\ 
      \hline
      \cellcolor{gray}{1} & 0 & 0 & \cellcolor{lightgray}{1} & \cellcolor{lightgray}{1} & \cellcolor{lightgray}{1} & \cellcolor{lightgray}{1} \\ 
      \hline
      \cellcolor{gray}{1} & 0 & 0 & \cellcolor{lightgray}{1} & 0 & \cellcolor{orange}{1} & \cellcolor{orange}{1} \\ 
      \hline
      \cellcolor{gray}{1} & 0 & \cellcolor{lightgray}{1} & \cellcolor{lightgray}{1} & 0 & \cellcolor{orange}{1} & \cellcolor{pink}{1} \\ 
      \hline
      \cellcolor{gray}{1} & \cellcolor{lightgray}{1} & \cellcolor{lightgray}{1} & 0 & \cellcolor{orange}{1} & \cellcolor{orange}{1} & \cellcolor{pink}{1} \\ 
      \hline
      \cellcolor{gray}{1} & \cellcolor{lightgray}{1} & \cellcolor{orange}{1} & \cellcolor{orange}{1} & \cellcolor{orange}{1} & 0 & \cellcolor{pink}{1} \\ 
      \hline
      \cellcolor{gray}{1} & \cellcolor{lightgray}{1} & \cellcolor{orange}{1} & 0 & 0 & \cellcolor{pink}{1} & \cellcolor{pink}{1} \\ 
      \hline
      \cellcolor{gray}{1} & \cellcolor{lightgray}{1} & \cellcolor{orange}{1} & 0 & \cellcolor{pink}{1} & \cellcolor{pink}{1} & 0 \\ 
      \hline
      \cellcolor{gray}{1} & \cellcolor{lightgray}{1} & \cellcolor{orange}{1} & \cellcolor{pink}{1} & \cellcolor{pink}{1} & 0 & 0 \\ 
    \end{array}
    \right]$}
    \hspace{10mm}
    \subfigure{
    $\left[
    \begin{array}{c|c|c|c|c|c|c}
      \cellcolor{gray}{1} & \cellcolor{gray}{1} & \cellcolor{gray}{1} & \cellcolor{gray}{1} & \cellcolor{gray}{1} & \cellcolor{gray}{1} & \cellcolor{gray}{1} \\ 
      \hline
      \cellcolor{gray}{1} & 0 & 0 & \cellcolor{orange}{1} & \cellcolor{orange}{1} & \cellcolor{orange}{1} & \cellcolor{lightgray}{1} \\ 
      \hline
      \cellcolor{gray}{1} & 0 & 0 & \cellcolor{orange}{1} & 0 & \cellcolor{pink}{1} & \cellcolor{lightgray}{1} \\ 
      \hline
      \cellcolor{gray}{1} & 0 & \cellcolor{orange}{1} & \cellcolor{orange}{1} & 0 & \cellcolor{pink}{1} & \cellcolor{lightgray}{1} \\ 
      \hline
      \cellcolor{gray}{1} & \cellcolor{orange}{1} & \cellcolor{orange}{1} & 0 & \cellcolor{pink}{1} & \cellcolor{pink}{1} & \cellcolor{lightgray}{1} \\ 
      \hline
      \cellcolor{gray}{1} & \cellcolor{orange}{1} & \cellcolor{pink}{1} & \cellcolor{pink}{1} & \cellcolor{pink}{1} & 0 & \cellcolor{lightgray}{1} \\ 
      \hline
      \cellcolor{gray}{1} & \cellcolor{orange}{1} & \cellcolor{pink}{1} & 0 & 0 & \cellcolor{lightgray}{1} & \cellcolor{lightgray}{1} \\ 
      \hline
      \cellcolor{gray}{1} & \cellcolor{orange}{1} & \cellcolor{pink}{1} & 0 & \cellcolor{lightgray}{1} & \cellcolor{lightgray}{1} & 0 \\ 
      \hline
      \cellcolor{gray}{1} & \cellcolor{lightgray}{1} & \cellcolor{lightgray}{1} & \cellcolor{lightgray}{1} & \cellcolor{lightgray}{1} & 0 & 0 \\ 
    \end{array}
    \right]$}
    \caption{Two different decompositions of a matrix in $\mathcal{M}_{9,7;5}$ into $4$ zigzag paths.}
    \label{fig.example}
\end{figure}

\begin{lemma}\label{lem.zigzagbij}
    The set of maximal $I_k$-avoiding $(0,1)$-matrices of size $m \times n$ $\mathcal{M}_{m,n;k}$ is in bijection with the set of zigzag paths $p_1^{\prime},p_2^{\prime},\dots,p_{k-1}^{\prime}$, where $p_i^{\prime}$ starts at the entry $\ell_i=(m-i+1,k-i)$ and ends at the entry $r_i=(k-i,n-i+1)$.
\end{lemma}
\begin{proof}
    From the above discussion and Lemma \ref{lem.zigzags}, it is clear that a matrix in $\mathcal{M}_{m,n;k}$ uniquely determines the $k-1$ zigzag paths starting at $\ell_i$'s and ending at $r_i$'s. 

    Conversely, given zigzag paths $p_1^{\prime},p_2^{\prime},\dots,p_{k-1}^{\prime}$, we first fill in the entries in the lower left and the upper right staircases of size $k-2$ with $1$'s. Next, observe that we cannot add any $1$'s in the matrix since this will create $k$ diagonal entries containing $1$'s. This gives a matrix in $\mathcal{M}_{m,n;k}$.
\end{proof}

We are now ready to prove that $\mathcal{M}_{m,n;k}$ is in bijection with plane partitions $\mathcal{PP}(n-k+1,m-k+1,k-1)$.
\begin{proof}[Proof of Theorem \ref{thm1}.]
    One can identify each entry of the $m \times n$ matrix $M$ as a lattice point on the square lattice of size $(m-1) \times (n-1)$; the horizontal edges are oriented east and the vertical edges are oriented north. By convention, the $(m,1)$-entry of a matrix is identified as the origin on the square lattice. These zigzag paths can be viewed as non-intersecting lattice paths that go from the bottom left to the top right on the square lattice, passing through lattice points that correspond to the entry $1$ of $M$.
    
    By Lemma \ref{lem.zigzagbij}, $\mathcal{M}_{m,n;k}$ is in bijection with the set of zigzag paths $p_1^{\prime},p_2^{\prime},\dots,p_{k-1}^{\prime}$, where $p_i^{\prime}$ starts at the entry $\ell_i=(m-i+1,k-i)$ and ends at the entry $r_i=(k-i,n-i+1)$. On the square lattice, these zigzag paths correspond to non-intersecting lattice paths $\mathcal{P}(U,V)$, where the set of starting points $U=\{u_i=(k-i-1,i-1)|i=1,\dots,k-1\}$ and the set of ending points $V=\{v_i=(n-i,m-k+i)|i=1,\dots,k-1\}$. Moreover, as introduced in Section \ref{sec.mainPP}, $\mathcal{P}(U,V)$ is in bijection with the set of plane partitions $\mathcal{PP}(a,b,c)$ with $a=m-k+1,b=n-k+1$, and $c=k-1$. We encourage the reader to check that the plane partition given in Figure \ref{fig.pp} corresponds to the matrix given in Figure \ref{fig.example} via the non-intersecting lattice paths shown in Figure \ref{fig.pplattice}.

    Finally, by \eqref{eq.Hproduct}, we obtain $|\mathcal{M}_{m,n;k}| = H(m-k+1,n-k+1,k-1)$. This completes the proof of Theorem \ref{thm1}.
\end{proof}

We close this subsection by analyzing the number of decompositions of each matrix into zigzag paths, which is presented in the following proposition.
\begin{proposition}\label{prop.ribbon}
    The number of decompositions of a matrix $M \in \mathcal{M}_{m,n;k}$ into $k-1$ zigzag paths introduced by Brualdi and Cao is $(k-1)!$.
\end{proposition}
\begin{proof}
    Based on the proof of Lemma \ref{lem.zigzags}, the number of the decompositions of $M$ is given by the number of selections of entries for each path in the staircases, as well as the choices of entries $\ell_{i_j}$ and $r_{i_j}$. Due to the length restriction of each zigzag path, once the zigzag paths in the lower left staircase are determined, the zigzag paths in the upper right staircase are automatically given by flipping the one in the lower left staircase across the diagonal. 
    
    Now, the problem is equivalent to finding the number of decompositions of the staircase of size $k-1$ into $k-1$ zigzag paths satisfying
    \begin{itemize}
        \item the length of each zigzag path is $1,2,\dots,k-1$, respectively, and
        \item each zigzag path must contain exactly one diagonal entry of the staircase.
    \end{itemize}
    Let $\mathcal{Z}_{k-1}$ be the set of these decompositions mentioned above. We claim that $|\mathcal{Z}_{k-1}| = (k-1)!$.

    To prove the claim, we proceed by induction on $k$. The base case, $k=2$, is clear since the staircase reduces to a single entry, and there is only one decomposition. Now, we assume that the statement holds for $k-2$, that is, $|\mathcal{Z}_{k-2}|=(k-2)!$.
    
    Let $L=\{\ell_1,\ell_2,\dots,\ell_{k-1}\}$ be the entries on the diagonal of the staircase of size $k-1$. To obtain a decomposition in $\mathcal{Z}_{k-1}$, we first need to choose one of the entries from $L$, say $\ell_j$, to form the zigzag path of length $1$. This gives $k-1$ choices. Second, we extend each zigzag path in a decomposition of $\mathcal{Z}_{k-2}$ by including one of the entries from $L \setminus \{\ell_j\}$. 
    
    Notice that in a decomposition of $\mathcal{Z}_{k-2}$, the two zigzag paths that contain the entry below $\ell_j$ and on the left of $\ell_j$, respectively, have only one way to extend them. Then, the next two adjacent zigzag paths in the same decomposition of $\mathcal{Z}_{k-2}$ again have only one way to extend them. Continuing this process, one can see that once $\ell_j$ is determined, there is a unique way to extend each decomposition of $\mathcal{Z}_{k-2}$. 
    \begin{figure}[hbt!]
        \centering
        \includegraphics[height=0.28\textwidth]{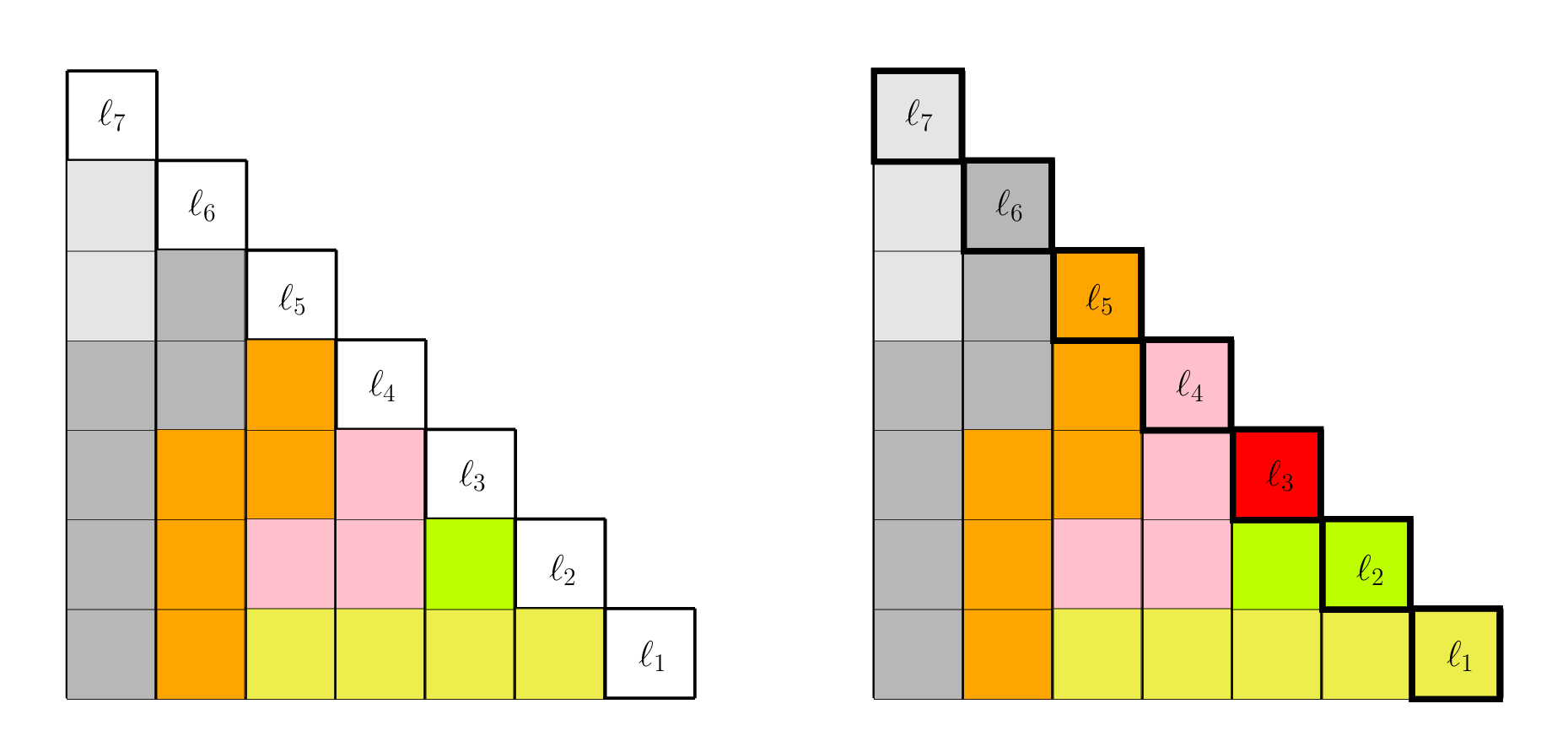}
        \caption{An example of extending zigzag paths in a decomposition of $\mathcal{Z}_6$ (left) to a decomposition of $\mathcal{Z}_7$ (right).}
        \label{fig.zigzagdecomp}
    \end{figure}

    Figure \ref{fig.zigzagdecomp} shows an example of extending a decomposition in $\mathcal{Z}_{6}$ to $\mathcal{Z}_7$ by selecting $\ell_3$ (the entry colored red) as the zigzag path of length $1$. Notice that the zigzag paths colored pink and lime in $\mathcal{Z}_6$ contain the entries below and to the left of $\ell_3$; there is only one way to extend these two zigzag paths. Therefore, $|\mathcal{Z}_{k-1}| = |\mathcal{Z}_{k-2}| \cdot (k-1) = (k-1)!$. 
\end{proof}

\subsection{The weighted enumeration of maximal $I_k$-avoiding $(0,1)$-matrices}\label{sec.mainweight}

We define two statistics on our matrix below. Let $M=[M_{i,j}]$ be a matrix in $\mathcal{M}_{m,n;k}$. For an entry $M_{i,j} = 0$, define
\begin{equation}
    \mathsf{v}(M_{i,j})=|\{d>0 | M_{i+d,j+d} = 1 \}|.
\end{equation}
That is, for each $M_{i,j}=0$, the value $\mathsf{v}(M_{i,j})$ counts the number of $1$'s lying strictly below and to the right of $M_{i,j}$ along its diagonal. We also define $\mathsf{v}(M)= \sum_{M_{i,j}=0} \mathsf{v}(M_{i,j})$ (resp., $\mathsf{v}_d(M) = \sum_{M_{i,i}=0} \mathsf{v}(M_{i,i})$) to be the total value of $\mathsf{v}(M_{i,j})$ for all (resp., diagonal) zero entries $M_{i,j}$ of $M$.

On the other hand, for an entry $M_{i,j} = 1$, define
\begin{equation}
    \mathsf{w}(M_{i,j}) =|\{d>0 | M_{i-d,j-d} = 0 \}|.
\end{equation}
That is, for each $M_{i,j}=1$, the value $\mathsf{w}(M_{i,j})$ counts the number of $0$'s lying strictly above and to the left of $M_{i,j}$ along its diagonal.

From Lemma \ref{lem.zigzagbij}, we assume the $j$th zigzag path $p_j^{\prime}$ meets the main diagonal of $M$ at $M_{i_j,i_j}$. For each $j=1,2.\dots,k-1$, define the sequence $\mathsf{d}_j(M) = \mathsf{w}(M_{i_j,i_j})$ (here we read the entries $M_{i_j,i_j}$ in the reverse order from lower right to upper left of $M$).


As a corollary of Theorem \ref{thm1}, we obtain the triple product formula for the generating function of matrices in $\mathcal{M}_{m,n;k}$ weighted based on the above two statistics. 
\begin{corollary}\label{cor.product}
    Given three positive integers $m,n$, and $k$ with $2 \leq k \leq \min\{m,n\}$. Then the generating function of matrices in $\mathcal{M}_{m,n;k}$ is given by
    \begin{equation}\label{eq.gf}
        \sum_{M \in \mathcal{M}_{m,n;k}}  q^{\mathsf{v}(M)} t^{\mathsf{v}_d(M)} \prod_{\ell = 1}^{k-1} \frac{(q^{k-\ell}  ; q)_{\mathsf{d}_{\ell}(M)}}{(tq^{k-\ell} ; q)_{\mathsf{d}_{\ell}(M)}} = \prod_{i=1}^{m-k+1} \prod_{j=1}^{n-k+1} \prod_{\ell=1}^{k-1} \frac{1-tq^{i+j+\ell-1}}{1-tq^{i+j+\ell-2}},
    \end{equation}
    where $(x;q)_N = \begin{cases} (1-x)(1-xq)(1-xq^2)\cdots (1-xq^{N-1}), &\text{ if $N \geq 1$,} \\
    1, &\text{ if $N=0$.}
    \end{cases}$
\end{corollary}

Before proving Corollary \ref{cor.product}, we note that when specializing $t=1$, \eqref{eq.gf} reduces to the following product formula
\begin{equation}\label{eq.matrixvolume}
    \sum_{M \in \mathcal{M}_{m,n;k}}  q^{\mathsf{v}(M)} = \prod_{i=1}^{m-k+1} \prod_{j=1}^{n-k+1} \prod_{\ell=1}^{k-1} \frac{1-q^{i+j+\ell-1}}{1-q^{i+j+\ell-2}}.
\end{equation}
Moreover, if we take $t=1$ and the limit $q \rightarrow 1$, then \eqref{eq.gf} simplifies to \eqref{eq.thm1}.
\begin{example}
    Consider maximal $I_3$-avoiding $(0,1)$-matrices of size $3 \times 4$. The six matrices, their corresponding statistics, and weights are listed in Table \ref{tab.gf}. Recall that $\mathsf{d}_1(M)$ and $\mathsf{d}_2(M)$ are given by $\mathsf{w}(M_{j,j})$ where the diagonal entries $M_{j,j}$ are marked red in each matrix. A direct calculation shows that the sum of these six terms equals
    \begin{equation}
        \frac{1-tq^3}{1-tq} \frac{1-tq^4}{1-tq^2}.
    \end{equation}
    This agrees with the right-hand side of \eqref{eq.gf} when $m=3,n=4$, and $k=3$.
    \begin{table}[htb!]
        \renewcommand{\arraystretch}{1.1}
        \centering
        \begin{tabular}{c|c|c|c|c}
          Matrix $M$  & $\mathsf{v}(M)$ & $\mathsf{v}_d(M)$ & $(\mathsf{d}_1(M),\mathsf{d}_2(M))$ & Weight \\
        \hline\hline
        $\left[
        \begin{array}{c|c|c|c}
          \textcolor{red}{1} & 1 & 1 & 1 \\
          \hline
          1 & \textcolor{red}{1} & 1 & 1 \\
          \hline
          1 & 1 & 0 & 0 
        \end{array}
        \right]$
        & $0$ & $0$ & $(0,0)$ & $1$ \\
        \hline\hline
        $\left[
        \begin{array}{c|c|c|c}
          \textcolor{red}{1} & 1 & 1 & 1 \\
          \hline
          1 & 0 & 1 & 1 \\
          \hline
          1 & 1 & \textcolor{red}{1} & 0 
        \end{array}
        \right]$ 
        & $1$ & $1$ & $(1,0)$ & $qt\frac{1-q^2}{1-tq^2}$ \\
        \hline\hline
        $\left[
        \begin{array}{c|c|c|c}
          \textcolor{red}{1} & 1 & 1 & 1 \\
          \hline
          1 & 0 & 0 & 1 \\
          \hline
          1 & 1 & \textcolor{red}{1} & 1 
        \end{array}
        \right]$  
        & $2$ & $1$ & $(1,0)$ & $q^2t\frac{1-q^2}{1-tq^2}$ \\
        \hline\hline
        $\left[
        \begin{array}{c|c|c|c}
          0 & 1 & 1 & 1 \\
          \hline
          1 & \textcolor{red}{1} & 1 & 1 \\
          \hline
          1 & 1 & \textcolor{red}{1} & 0 
        \end{array}
        \right]$
        & $2$ & $2$ & $(1,1)$ & $q^2t^2\frac{1-q^2}{1-tq^2}\frac{1-q}{1-tq}$ \\
        \hline\hline
        $\left[
        \begin{array}{c|c|c|c}
          0 & 1 & 1 & 1 \\
          \hline
          1 & \textcolor{red}{1} & 0 & 1 \\
          \hline
          1 & 1 & \textcolor{red}{1} & 1 
        \end{array}
        \right]$
        & $3$ & $2$ & $(1,1)$ & $q^3t^2\frac{1-q^2}{1-tq^2}\frac{1-q}{1-tq}$ \\
        \hline\hline
        $\left[
        \begin{array}{c|c|c|c}
          0 & 0 & 1 & 1 \\
          \hline
          1 & \textcolor{red}{1} & 1 & 1 \\
          \hline
          1 & 1 & \textcolor{red}{1} & 1 
        \end{array}
        \right]$
        & $4$ & $2$ & $(1,1)$ & $q^4t^2\frac{1-q^2}{1-tq^2}\frac{1-q}{1-tq}$ \\
        \end{tabular}
        \vspace{2mm}
        \caption{The six matrices in $\mathcal{M}_{3,4;3}$, their corresponding statistics, and the weights in the generating function.}
        \label{tab.gf}
    \end{table}
\end{example}

\begin{proof}[Proof of Corollary \ref{cor.product}]
    By Theorem \ref{thm1}, there is a bijection between $\mathcal{M}_{m,n;k}$ and $\mathcal{PP}(m-k+1,n-k+1,k-1)$ via non-intersecting lattice paths. Under this bijection, the key observation is that the zero entries in a matrix $M \in \mathcal{M}_{m,n;k}$ correspond to the top unit squares of a pile of unit cubes (where paths do not traverse these unit squares). For a zero entry $M_{i,j}=0$ of $M$, the statistic $\mathsf{v}(M_{i,j})$ counts the number of paths ``below'' it, which is exactly the number of unit cubes in that stack of unit cubes. It turns out that $\mathsf{v}(M)$ counts the total number of unit cubes of the plane partition $\pi$ corresponding to $M$, this is called the \emph{volume} of $\pi$, denoted by $\mathsf{vol}(\pi) = \sum_{i,j} \pi_{i,j}$. Similarly, the quantity $\mathsf{v}_{d}(M)$ sums over the diagonal entries of $M$, this is called the \emph{trace} of $\pi$, denoted by $\mathsf{tr}(\pi) = \sum_{i} \pi_{i,i}$. 

    Next, we assume the zigzag path $p_j^{\prime}$ meets the diagonal entry of $M$ at $M_{i_j,i_j}$. It is clear that, under the bijection, $p_j^{\prime}$ depicts the shape of the $j$th layer of the corresponding plane partition $\pi$, that is, $\lambda^j(\pi)$. Moreover, the statistic $\mathsf{d}_{j}(M)=\mathsf{w}(M_{i_j,i_j})$ counts the size of the \emph{Durfee square}\footnote{A \emph{Durfee square} of a partition $\lambda$ is the maximal square contained in the Young diagram of $\lambda$.} of the partition $\lambda^j(\pi)$. 

    Our result follows directly from the triple product formula for the generating function of plane partitions with bounded parts due to Kamioka \cite[Theorem 9]{Kamioka15}, which is given by
    \begin{equation}\label{eq.Kamioka}
        \sum_{\pi \in \mathcal{PP}(a,b,c)}  q^{\mathsf{vol}(\pi)} t^{\mathsf{tr}(\pi)} \prod_{\ell = 1}^{\pi_{1,1}} \frac{(q^{c+1-\ell}  ; q)_{\mathsf{d}_{\ell}(\pi)}}{(tq^{c+1-\ell} ; q)_{\mathsf{d}_{\ell}(\pi)}} = \prod_{i=1}^{a} \prod_{j=1}^{b} \prod_{\ell=1}^{c} \frac{1-tq^{i+j+\ell-1}}{1-tq^{i+j+\ell-2}},
    \end{equation}
    where $\mathsf{d}_{\ell}(\pi)$ is the size of the Durfee square of the partition $\lambda^{\ell}(\pi)$.
    
    We remark that the upper range $\pi_{1,1}$ of the product on the left-hand side of \eqref{eq.Kamioka} can be replaced by $c$. If the largest part $\pi_{1,1}$ is less than the height $c$ of the box, then for $\pi_{1,1} < j \leq c$, the $j$th layer of this plane partition is empty. Thus, the quantity $\mathsf{d}_{j}(\pi) = 0$, which contributes $1$ in the product. 

    Based on the above discussion, taking $a=m-k+1$, $b=n-k+1$, and $c=k-1$ in \eqref{eq.Kamioka} completes the proof.   
\end{proof}
\begin{remark}
    Note that specializing $t=1$, \eqref{eq.matrixvolume} is a counterpart of the classical result of MacMahon \cite{MacM}. He showed the following elegant product formula for the volume generating function of $\mathcal{PP}(a,b,c)$:
    \begin{equation}\label{eq.PPvolume}
        \sum_{\pi \in PP(a,b,c)} q^{\mathsf{vol}(\pi)}
         = \prod_{i=1}^{a} \prod_{j=1}^{b} \prod_{\ell=1}^{c} \frac{1-q^{i+j+\ell-1}}{1-q^{i+j+\ell-2}}.
    \end{equation}
    The generating function concerning both the volume and trace of plane partitions has been studied by Stanley \cite[Section 19]{Stan71} back to the 1970s. He proved the following formula for unbounded plane partitions $\mathcal{PP}(a,b,\infty)$, that is, the original condition $\pi_{i,j} \leq c$ is omitted. 
    \begin{equation}\label{eq.volumetrace}
        \sum_{\pi \in PP(a,b,\infty)} q^{\mathsf{vol}(\pi)}t^{\mathsf{tr}(\pi)} = \prod_{i=1}^{a} \prod_{j=1}^{b} \frac{1}{1-tq^{i+j-1}}.
    \end{equation}
    However, we do not have an analogue result of matrices in $\mathcal{M}_{m,n;k}$. Under the bijection between $\mathcal{PP}(a,b,c)$ and $\mathcal{M}_{m,n;k}$, as $c$ goes to infinity, both $m = a+c,n=b+c$, and $k=c+1$ approach infinity. Studying infinite matrices avoiding an infinite pattern is not the main focus of this paper, this direction is left to the interested reader.
\end{remark}
\begin{remark}
    We would like to point out that Gansner \cite{Gansner81-1,Gansner81-2} extended Stanley's generating function \eqref{eq.volumetrace} of unbounded plane partitions by introducing \emph{$\ell$-traces}, that is, the trace on the $\ell$th diagonal of a plane partition. Kamioka (\cite[Theorem 17]{Kamioka15} and \cite[Theorem 9]{Kamioka17}) also showed a refined version of \eqref{eq.Kamioka} by considering $\ell$-traces. Under our bijection, we can obtain a similar product formula for the generating function by including the idea of $\ell$-traces to matrices in $\mathcal{M}_{m,n;k}$. However, the weight presented by Kamioka (the term $w_{r,n}(\pi)$ in \cite[Theorem 9]{Kamioka17}) looks complicated and unnatural so we do not present an analogue result for matrices in $\mathcal{M}_{m,n;k}$.
\end{remark}

\section{Symmetry classes}\label{sec.sym}

In Theorem \ref{thm1}, we showed that maximal $I_k$-avoiding $(0,1)$-matrices of size $m \times n$ (IAMs) are equinumerous with plane partitions of size $(m-k+1) \times (n-k+1) \times (k-1)$. However, the structure of their symmetry classes is different. We first recall the ten symmetry classes of plane partitions in Section \ref{sec.ppsym}. The proof of Theorem \ref{thm2} is given in Section \ref{sec.pfthm2}.

\subsection{Ten symmetry classes of plane partitions}\label{sec.ppsym}

The study of symmetry classes of plane partitions dates back to MacMahon \cite{Mac1899} and gained further attention when Stanley \cite{Stan86pp} identified ten symmetry classes through the following three operations on plane partitions. Let $\pi=(\pi_{i,j}) \in PP(a,b,c)$ be a plane partition.
\begin{itemize}
    \item \emph{Reflection} of $\pi$: $\mathsf{re}(\pi)$ reflects along the main diagonal of $\pi$, that is, $\mathsf{re}(\pi_{i,j})=\pi_{j,i}$ for all $i$ and $j$.

    \item \emph{Rotation by $120^{\circ}$} of $\pi$: $\mathsf{ro}(\pi)$ is given by viewing $\pi$ as a pile of unit cubes and then rotating it by $120^{\circ}$ with the rotational axis $\{(t,t,t)| t \in \mathbb{R}\}$.

    \item \emph{Complement} of $\pi$: $\mathsf{co}(\pi_{i,j}) = c-\pi_{a+1-i,b+1-j}$. If we view $\pi$ as a pile of cubes, then taking the complement of $\pi$ is the set-theoretic complement inside the $a \times b \times c$ box.
\end{itemize}
These three operations generate a group which is isomorphic to the dihedral group $D_{12}$ of order $12$. The ten conjugacy classes of subgroups of $D_{12}$ corresponding to ten symmetry classes of plane partitions are shown in the following list.
\begin{itemize}
    \item[(U)] \emph{Unrestricted} plane partitions.
    
    \item[(S)] \emph{Symmetric} plane partitions that are invariant under the reflection along the main diagonal.
    
    \item[(CS)] \emph{Cyclically symmetric} plane partitions that are invariant under a $120^{\circ}$ rotation.
    
    \item[(TS)] \emph{Totally symmetric} plane partitions that are both symmetric and cyclically symmetric.
    
    \item[(SC)] \emph{Self-complementary} plane partitions that are invariant under taking the complement.
    
    \item[(TC)] \emph{Transpose-complementary} plane partitions that are equal to the reflection of their complement.
    
    \item[(SSC)] \emph{Symmetric Self-complementary} plane partitions that are symmetric and self-complementary.

    \item[(CSTC)] \emph{Cyclically symmetric transpose-complementary} plane partitions that are both cyclically symmetric and transpose-complementary.

    \item[(CSSC)] \emph{Cyclically symmetric self-complementary} plane partitions that are both cyclically symmetric and self-complementary.

    \item[(TSSC)] \emph{Totally symmetric self-complementary} plane partitions that are both symmetric, cyclically symmetric, and self-complementary.
\end{itemize}
For each symmetry class $* = U, S, CS,\dots$, let $\mathcal{PP}^{*}(a,b,c)$ denote the set of plane partitions of size $a \times b \times c$ under the symmetry $*$.

Surprisingly, the enumeration formula of each symmetry class of plane partitions has a simple product formula. It took mathematicians decades of effort to make a conjecture and prove these formulas. For more details, we refer the reader to \cite[Section 6]{Kra15S} and references therein.

\subsection{Ten symmetry classes of IAMs}\label{sec.pfthm2}

The proof of Theorem \ref{thm2} is presented below.
\begin{proof}[Proof of Theorem \ref{thm2}.]
By Theorem \ref{thm1}, we will view a matrix $M \in \mathcal{M}_{m,n;k}$ as both non-intersecting lattice paths and plane partitions in the proof.
\begin{enumerate}
    \item The unrestricted case follows directly from \eqref{eq.thm1}.
    
    \item The diagonally symmetric IAMs exist when $m=n$. Under the bijection, it is easy to see that the matrices in this class correspond to non-intersecting paths on the square lattice that are symmetric about the line $x+y=n-1$. Thus, they correspond to symmetric plane partitions of size $(n-k+1) \times (n-k+1) \times (k-1)$. Therefore, 
    \begin{equation}
        |\mathcal{M}^{DS}_{n,n;k}| = |\mathcal{PP}^{S}(n-k+1,n-k+1,k-1)| = \prod_{1 \leq i \leq j \leq n-k+1} \frac{k+i+j-2}{i+j-1},
    \end{equation}
    where the last equality was conjectured by MacMahon \cite{Mac1899} and proved by Andrews \cite{Andrews77}.

    \item For anti-diagonally symmetric IAMs, one requires $m=n$ and $k$ is odd. If $M$ is an anti-diagonally symmetric IAM and avoids $I_{k}$, where $k=2k^{\prime}$, then $M$ is in bijection with $2k^{\prime}-1$ non-intersecting lattice paths on the square lattice which are symmetric about the line $y=x$. This is impossible because no path can start and end at the points on the axis of symmetry. 

    Assume $k$ is odd. In this class, the entry $M_{i,j}$ is zero if and only if the entry $M_{n+1-j,n+1-i}$ is zero. Assume $\mathsf{v}(M_{n+1-j,n+1-i}) = x$, that is, there are $x$ zigzag paths that pass through the lower right of $M_{n+1-j,n+1-i}$. Due to symmetry, there are $x$ zigzag paths that pass through the upper left of $M_{i,j}$, and this results in $k-1-2x$ zigzag paths passing through the middle of these two entries. So, $\mathsf{v}(M_{i,j}) + \mathsf{v}(M_{n+1-j,n+1-i}) = k-1$. One can see that the corresponding entries in the plane partition have the property that they are equal to the reflection of their complement. In other words, the IAMs in this class correspond to the transpose-complementary plane partitions. 
    \begin{align}
        |\mathcal{M}^{AS}_{n,n;k}| & = |\mathcal{PP}^{TC}(n-k+1,n-k+1,k-1)| \nonumber\\
        &= \displaystyle 
                \binom{n-\frac{k+1}{2}}{n-k}\prod_{1 \leq i \leq j \leq n-k-1} \frac{k+i+j}{i+j+1},
    \end{align}
    where the last equality was proved by Proctor \cite{Proc84}.

    \item The diagonally and anti-diagonally symmetric IAMs are the intersection of the previous two classes, one requires $m=n$ and $k$ is odd. The corresponding plane partitions lie in the intersection of the set of symmetric plane partitions and the set of transpose-complementary plane partitions. We claim that their intersection is the set of symmetric self-complementary plane partitions.
    \begin{equation}
        \mathcal{PP}^{S}(a,a,2c) \cap \mathcal{PP}^{TC}(a,a,2c) = \mathcal{PP}^{SSC}(a,a,2c).
    \end{equation}
    If $\pi \in \mathcal{PP}^{S}(a,a,2c) \cap \mathcal{PP}^{TC}(a,a,2c)$, then $\mathsf{re}(\pi) = \pi$ and $\mathsf{co}(\pi) = \mathsf{re}(\mathsf{re}(\mathsf{co}(\pi))) = \mathsf{re}(\pi) = \pi$. So, $\pi \in \mathcal{PP}^{SSC}(a,a,2c)$. The other direction is obvious. 

    Finally, due to Proctor \cite{Proc83} again, we obtain for $k$ odd,
    \begin{align}
        |\mathcal{M}_{n,n;k}^{DAS}| & = |\mathcal{PP}^{SSC}(n-k+1,n-k+1,k-1)| \nonumber \\
        & =\begin{cases}
            H(\frac{n-k+2}{2}, \frac{n-k}{2}, \frac{k-1}{2})    , &\text{ if $n$ is odd,}\\
            H(\frac{n-k+1}{2}, \frac{n-k+1}{2}, \frac{k-1}{2})    , &\text{ if $n$ is even.}
            \end{cases} 
    \end{align}

    \item For half-turn symmetric IAMs, the entry $M_{i,j}$ is zero if and only if the entry $M_{m+1-i,n+1-j}$ is zero. Also, the corresponding non-intersecting lattice paths are invariant under a rotation by $180^{\circ}$. By a similar argument for the anti-diagonally symmetric IAMs, we obtain $\mathsf{v}(M_{i,j}) + \mathsf{v}(M_{m+1-i,n+1-j}) = k-1$. One can see that the corresponding entries in the plane partition have the property that they are equal to their complement. Hence, the IAMs in this class correspond to the self-complementary plane partitions. 

    The enumeration of self-complementary plane partitions was proved by Stanley \cite{Stan86pp},
    \begin{equation}\label{eq.ppsymSC}
        |\mathcal{PP}^{SC}(a,b,c)| = \begin{cases}
                H(\frac{a}{2},\frac{b}{2},\frac{c}{2})^2 , &\text{ for $a,b,c$ even,} \\
                H(\frac{a}{2},\frac{b+1}{2},\frac{c-1}{2})H(\frac{a}{2},\frac{b-1}{2},\frac{c+1}{2}) , &\text{ for $a$ even and $b,c$ odd,} \\
                H(\frac{a+1}{2},\frac{b}{2},\frac{c}{2})H(\frac{a-1}{2},\frac{b}{2},\frac{c}{2}) , &\text{ for $a$ odd and $b,c$ even.}
                \end{cases}
    \end{equation}
    Note that self-complementary plane partitions contain exactly half of the unit cubes in the box, so there is no self-complementary plane partition with $a,b,c$ odd.
    
    In our case, we need to find the parities of three numbers $m-k+1,n-k+1$, and $k-1$, and then obtain the counting formula by substituting $a,b,c$ carefully in \eqref{eq.ppsymSC}. Note that permuting the parameters $x,y,z$ does not change the value of the function $H(x,y,z)$. For example, if $k$ is even and $m,n$ are odd, then $m-k+1,n-k+1$ are even and $k-1$ is odd. We use the third equation in \eqref{eq.ppsymSC}, taking $a=k-1,b=n-k+1,c=m-k+1$ and rearranging these three parameters yield the first case in \eqref{eq.symHTSeven}. The remaining cases can be shown similarly; their arguments are omitted here.

    \item For the remaining symmetry classes, let $M \in \mathcal{M}^{*}_{m,n;k}$, where $* = HS, VS, VHS, QTS$, or $TS$. Due to symmetry, the four staircases of size $k-1$ located at the four corners of $M$ must be filled with $1$'s. The zigzag paths must contain the $1$'s in the upper left and lower right staircases. This results in only one way to arrange these zigzag paths, and the number of such zigzag paths must be even (see Figure \ref{fig.matrixsym} for an illustration). This implies that $k$ needs to be an odd number. 
    
    Therefore, 
    \begin{equation}
        |\mathcal{M}_{m,n;k}^{VS}|=|\mathcal{M}_{m,n;k}^{HS}| = |\mathcal{M}_{m,n;k}^{VHS}| = |\mathcal{M}_{n,n;k}^{QTS}| = |\mathcal{M}_{n,n;k}^{TS}| = \begin{cases}
            1, & \text{if $k$ is odd,}\\
            0, & \text{if $k$ is even.}
        \end{cases}.
    \end{equation}

    \begin{figure}[hbt!]
        \centering
        $\left[
            \begin{array}{c|c|c|c|c|c|c|c|c|c|c|c}
             \cellcolor{gray}{\textcolor{blue}{1}} & \cellcolor{gray}{\textcolor{blue}{1}} & \cellcolor{gray}{\textcolor{blue}{1}} & \cellcolor{gray}{\textcolor{blue}{1}} & \cellcolor{gray}{1} & \cellcolor{gray}{1} & \cellcolor{gray}{1} & \cellcolor{gray}{1} & \cellcolor{gray}{\textcolor{blue}{1}} & \textcolor{blue}{1} & \textcolor{blue}{1} & \textcolor{blue}{1} \\
             \hline
             \cellcolor{gray}{\textcolor{blue}{1}} & \cellcolor{lightgray}{\textcolor{blue}{1}} & \cellcolor{lightgray}{\textcolor{blue}{1}} & \cellcolor{lightgray}{1} & \cellcolor{lightgray}{1} & \cellcolor{lightgray}{1} & \cellcolor{lightgray}{1} & \cellcolor{lightgray}{1} & \cellcolor{lightgray}{1} & \cellcolor{lightgray}{\textcolor{blue}{1}} & \textcolor{blue}{1} & \textcolor{blue}{1} \\
             \hline
             \cellcolor{gray}{\textcolor{blue}{1}} & \cellcolor{lightgray}{\textcolor{blue}{1}} & 0 & 0 & 0 & 0 & 0 & 0 & 0 & 0 & \cellcolor{orange}{\textcolor{blue}{1}} & \textcolor{blue}{1} \\
             \hline
             \cellcolor{gray}{\textcolor{blue}{1}} & \cellcolor{lightgray}{1} & 0 & 0 & 0 & 0 & 0 & 0 & 0 & 0 & \cellcolor{orange}{1} & \cellcolor{pink}{\textcolor{blue}{1}} \\
             \hline
             \cellcolor{gray}{1} & \cellcolor{lightgray}{1} & 0 & 0 & 0 & 0 & 0 & 0 & 0 & 0 & \cellcolor{orange}{1} & \cellcolor{pink}{1} \\
             \hline
             \cellcolor{gray}{\textcolor{blue}{1}} & \cellcolor{lightgray}{1} & 0 & 0 & 0 & 0 & 0 & 0 & 0 & 0 & \cellcolor{orange}{1} & \cellcolor{pink}{\textcolor{blue}{1}} \\
             \hline
             \textcolor{blue}{1} & \cellcolor{lightgray}{\textcolor{blue}{1}} & 0 & 0 & 0 & 0 & 0 & 0 & 0 & 0 & \cellcolor{orange}{\textcolor{blue}{1}} & \cellcolor{pink}{\textcolor{blue}{1}} \\
             \hline
             \textcolor{blue}{1} & \textcolor{blue}{1} & \cellcolor{orange}{\textcolor{blue}{1}} & \cellcolor{orange}{1} & \cellcolor{orange}{1} & \cellcolor{orange}{1} & \cellcolor{orange}{1} & \cellcolor{orange}{1} & \cellcolor{orange}{1} & \cellcolor{orange}{\textcolor{blue}{1}} & \cellcolor{orange}{\textcolor{blue}{1}} & \cellcolor{pink}{\textcolor{blue}{1}} \\
             \hline
             \textcolor{blue}{1} & \textcolor{blue}{1} & \textcolor{blue}{1} & \cellcolor{pink}{\textcolor{blue}{1}} & \cellcolor{pink}{1} & \cellcolor{pink}{1} & \cellcolor{pink}{1} & \cellcolor{pink}{1} & \cellcolor{pink}{\textcolor{blue}{1}} & \cellcolor{pink}{\textcolor{blue}{1}} & \cellcolor{pink}{\textcolor{blue}{1}} & \cellcolor{pink}{\textcolor{blue}{1}}
            \end{array}
        \right]$
        \caption{A matrix in $\mathcal{M}^{*}_{9,12;5}$, where $* = HS, VS, VHS, QTS$, or $TS$. The $1$'s in the four staircases are marked blue. The four zigzag paths are presented in four different colors.}
        \label{fig.matrixsym}
    \end{figure}
    \end{enumerate}
    This completes the proof of Theorem \ref{thm2}. 
\end{proof}

In the work of Ciucu \cite{Ciucu16}, four intriguing product relations for the number of symmetry classes of plane partitions have been discovered. We close this section by pointing out similar product relations of the numbers of five symmetry classes of IAMs.
\begin{corollary}\label{cor.productrelation}
    Let $n,k$ be positive integers with $2 \leq 2k-1 \leq n$. Then
    \begin{enumerate}
        \item For unrestricted IAMs, diagonally symmetric IAMs, and anti-diagonally symmetric IAMs, they satisfy
        \begin{equation}\label{eq.symprod1}
            |\mathcal{M}^{U}_{n,n;2k-1}| = |\mathcal{M}^{DS}_{n,n;2k-1}| \cdot |\mathcal{M}^{AS}_{n,n;2k-1}|.
        \end{equation}

        \item For half-turn symmetric IAMs, and diagonally and anti-diagonally symmetric IAMs, they satisfy
        \begin{equation}\label{eq.symprod2}
            |\mathcal{M}^{HTS}_{n,n;2k-1}| = |\mathcal{M}^{DAS}_{n,n;2k-1}|^2.
        \end{equation}  
    \end{enumerate}
\end{corollary}
These relations can be verified directly from the formula presented in Theorem \ref{thm2}, the proof of Corollary \ref{cor.productrelation} is left to the reader. However, no bijective proof is currently known for the product relations among the symmetry classes of both plane partitions and IAMs. We hope the zigzag path decomposition of a matrix studied in this paper could bring new ideas to establish a bijection. An open problem is proposed.
\begin{problem}\label{prob1}
    Find a bijective proof of the equations \eqref{eq.symprod1} and \eqref{eq.symprod2}.
\end{problem}

\section{Maximal $I_k$-avoiding $(0,1)$-fillings of skew shapes}\label{sec.skew}

As an application of Theorem \ref{thm1}, we are able to enumerate the set $\mathcal{F}_{R;k}$ consisting of maximal $I_k$-avoiding $(0,1)$-fillings of the region $R$. In Section \ref{sec.maxstaircase}, we prove Theorem \ref{thm3} which states that $|\mathcal{F}_{R;k}|$ is given by the product formula, where $R = \bar{R}_{m,n;t}$ is obtained from a rectangle of size $m \times n$ by removing a maximal staircase of size $t$. In Section \ref{sec.skewshape}, we prove the conceptual formula for $|\mathcal{F}_{R;k}|$ in Theorem \ref{thm4}, where $R = \lambda/\mu$ is the Young diagram of a skew shape. 

As discussed in Section \ref{sec.mainIAM}, for each maximal $I_k$-avoiding filling of a rectangle, the entries forming the staircases of size $k-1$ located at the lower left and upper right corners of the region must be filled with $1$'s. Visualizing $1$'s as non-intersecting lattice paths, the entries along the zigzag edges of these staircases determine the endpoints of these lattice paths. Fixing the endpoints of these lattice paths, each path may go freely within the region, subject only to the non-intersecting condition. Therefore, if the region $R$ is obtained from a rectangle by deleting some entries at its upper left or lower right corner, we can still view maximal $I_k$-avoiding $(0,1)$-fillings of $R$ as families of non-intersecting paths, where each path is forbidden from passing through the deleted parts of $R$. Note that such a region $R$ can be naturally identified with the Young diagram of skew shape $\lambda/\mu$. We would like to point out that once $k$ is given, the entries forming the staircases of size $k$ must be able to fit inside the lower left and upper right corners of $R$. For skew shapes $\lambda/\mu$, this is equivalent to the conditions \eqref{eq.skewcondition}.

\subsection{Proof of Theorem \ref{thm3}}\label{sec.maxstaircase}

For the region $R=\bar{R}_{m,n;t}$ (assuming $m \leq n$), we consider the cases when $t=m-k$ or $m-k+1$, that is, we delete the staircase of maximal size from the corner. See Figure \ref{fig.removemaximal1} for a maximal $I_5$-avoiding $(0,1)$-filling of the $9 \times 12$ rectangle with the staircase of size $4$ removed from its lower right corner.

\begin{figure}[hbt!]
    \centering
    \subfigure[]{\label{fig.removemaximal1}\includegraphics[height=0.26\textwidth]{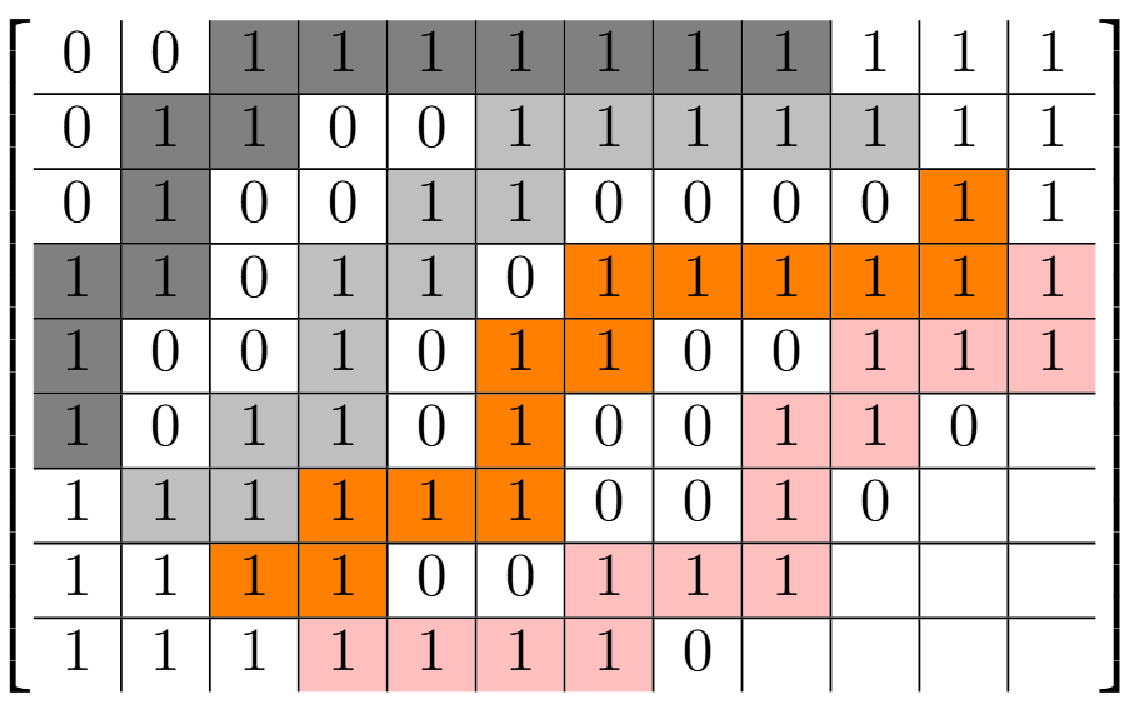}}
    \hspace{8mm}
    \subfigure[]{\label{fig.removemaximal2}
    \includegraphics[height=0.26\textwidth]{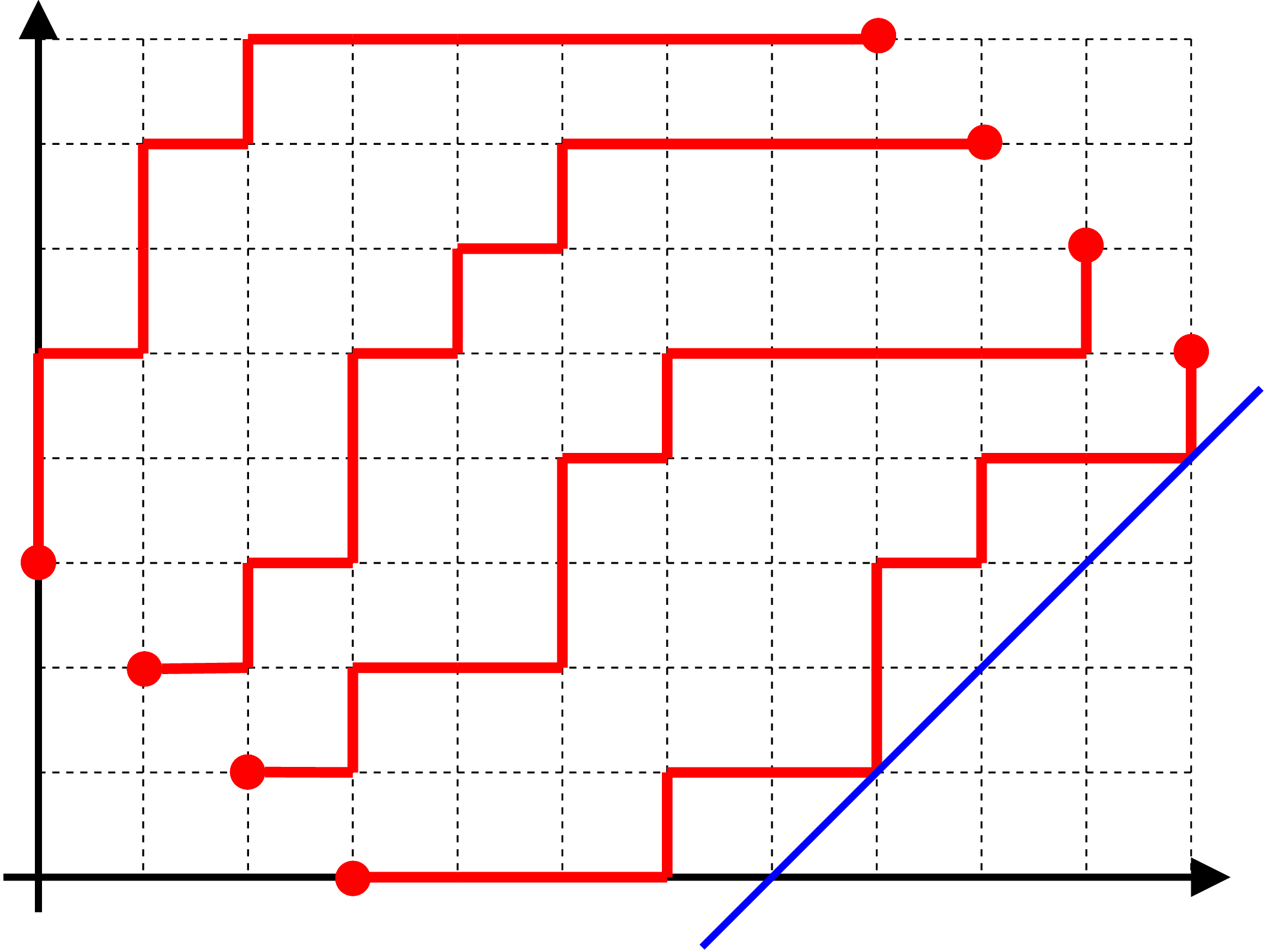}}
    \caption{(a) A maximal $I_5$-avoiding $(0,1)$-filling of $\bar{R}_{9,12;t}$, where $t=9-5$. (b) The corresponding non-intersecting lattice paths. Note that paths can only stay weakly above the blue line.}
    \label{fig.removemaximal}
\end{figure}

To prove Theorem \ref{thm3}, we need the following determinant evaluation presented by Krattenthaler; we combine two equalities and set $q=1$ in \cite[Theorem 30]{Det1}.
\begin{lemma}[Krattenthaler \cite{Det1}]\label{lem.det}
    Let $d$ be a nonnegative integer, and let $L_1,L_2,\dots,L_d$ and $A$ be indeterminates. Set $c \in \{0,1\}$. Then
    \begin{align}\label{eq.det1}
       \quad & \det \left[ \binom{A}{j-L_i} - \binom{A}{-j-L_i+c} \right]_{i,j=1}^{d} \nonumber \\ = &\prod_{i=1}^{d}\frac{(A+2i-1-c)!}{(d-L_i)!(A+d-c+L_i)!}\prod_{1 \leq i < j \leq d}(L_j-L_i) \prod_{1 \leq i \leq j \leq d}(L_i+L_j+A-c).
    \end{align}
\end{lemma}

The proof of Theorem \ref{thm3} is given below.
\begin{proof}[Proof of Theorem \ref{thm3}]
    We assume $t=m-k$ or $m-k+1$. Using the idea of the proof of Theorem \ref{thm1} and the discussion at the beginning of this section, the set $\mathcal{F}_{\bar{R}_{m,n;t};k}$ of maximal $I_k$-avoiding $(0,1)$-fillings of $\bar{R}_{m,n;t}$ is in bijection with the set of $k-1$ non-intersecting lattice paths $(p^{\prime}_1,\dots,p_{k-1}^{\prime})$ with the following properties.
    \begin{itemize}
        \item $p_i^{\prime}$ starts at $u_i = (k-i-1,i-1)$ and ends at $v_i = (n-i,m-k+i)$ for each $i=1,2,\dots,k-1$, and
        \item $p_i^{\prime}$'s are lattice paths in the square grid graph that stay weakly above the line $y=x+(t-n+1)$; see Figure \ref{fig.removemaximal2} for an illustration.
    \end{itemize}

    To count the number of such non-intersecting paths with the starting points $\{u_1,\dots,u_{k-1}\}$ and the ending points $\{v_1,\dots,v_{k-1}\}$, the Lindstr\"om--Gessel--Viennot (\cite{Lin73} and \cite{GV89,GV85}) states that this number is given by the determinant of a matrix $B=[b_{i,j}]$ of size $k-1$, where $b_{i,j}$ is the number of paths going from $u_i$ to $v_j$. Since our paths stay weakly above the line $y=x+(t-n+1)$, using the classical result of the \emph{reflection principle} \cite{andre1887}, the number of paths going from $u_i$ to $v_j$ is
    \begin{equation}\label{eq.reflection}
        b_{i,j} = \binom{m+n-2k+2}{m-k-i+j+1} - \binom{m+n-2k+2}{m-k-i-j+2-\delta_{t,n-k}}.
    \end{equation}

    The proof will be completed once we evaluate $\det B$. By Lemma \ref{lem.det}, taking $d=k-1$, $A=m+n-2k+2$, $L_i = -(m-k-i+1)$, $c=1-\delta_{t,m-k}$ in \eqref{eq.det1}, and noting that the middle product in \eqref{eq.det1} reads 
    \begin{equation*}
    \prod_{1 \leq i < j \leq k-1} (L_j-L_i) = \prod_{1 \leq i < j \leq k-1} (j-i) = \prod_{i=1}^{k-2} i!,
    \end{equation*}
    we then obtain the desired product formula \eqref{eq.maxstaircaseprod}. This completes the proof of Theorem \ref{thm3}.    
\end{proof}

\subsection{Proof of Theorem \ref{thm4}}\label{sec.skewshape}

We consider the region $R = \lambda/\mu$ to be a skew shape, where $\lambda$ and $\mu$ satisfy the conditions \eqref{eq.skewcondition}. Recall that if we identify each box in the Young diagram of $\lambda/\mu$ as a vertex, and two boxes are adjacent if and only if the two corresponding vertices are connected by an edge, then the resulting graph is the planar dual graph of $\lambda/\mu$; under our setting, such a graph is also the Young diagram of the skew shape denoted by $\bar{\lambda}/\bar{\mu}$.

To prove Theorem \ref{thm4}, we need the following lemma due to Kreweras \cite{Kreweras65} which provides a determinant formula for the number of lattice paths contained in the Young diagram of $\lambda/\mu$ from its lower left corner to its upper right corner, see also \cite[Theorem 10.7.1]{Kra15L} for a slightly different setting. 
\begin{lemma}[Kreweras \cite{Kreweras65}] \label{lem.skewshapepath}
    Let $n$ be the number of parts of $\lambda/\mu$. Then the number of lattice paths contained in the Young diagram of $\lambda/\mu$ from its lower left corner to its upper right corner is given by \eqref{eq.f},
    \begin{equation}
        f(\lambda,\mu) = \det \left[ \binom{\lambda_j-\mu_i+1}{j-i+1} \right]_{i,j=1}^{n}. 
    \end{equation}
\end{lemma}

Now, we are ready to prove Theorem \ref{thm4}. 
\begin{proof}[Proof of Theorem \ref{thm4}]
    Using the idea of the proof of Theorem \ref{thm1} and the discussion at the beginning of this section, the set $\mathcal{F}_{\lambda/\mu;k}$ of maximal $I_k$-avoiding $(0,1)$-fillings of $\lambda/\mu$ is in bijection with the set of $k-1$ non-intersecting lattice paths $(p_1^{\prime},\dots,p_{k-1}^{\prime})$ with the following properties.
    \begin{itemize}
        \item $p_i^{\prime}$ starts at $u_i = (k-i-1,i-1)$ and ends at $v_i = (n-i,m-k+i)$ for each $i=1,2,\dots,k-1$, and
        \item $p_i^{\prime}$'s are lattice paths in the planar dual graph $\bar{\lambda}/\bar{\mu}$.
    \end{itemize}

    For each lattice path $p_i^{\prime}$, we may extend the starting point from $u_i$ to $u^{\prime}_i=(0,i-1)$ and the ending point from $v_i$ to $v^{\prime}_i=(n-1,m-k+i)$. This is due to the fact that there is only one non-intersecting path connecting $\{u_1^{\prime},\dots,u_{k-1}^{\prime}\}$ with $\{u_1,\dots,u_{k-1}\}$; each path consists of horizontal steps only. Similarly, there is only one non-intersecting path connecting $\{v_1,\dots,v_{k-1}\}$ with $\{v_1^{\prime},\dots,v_{k-1}^{\prime}\}$. We may extend the first property above to
    \begin{itemize}
        \item $p_i^{\prime}$ starts at $u^{\prime}_i = (0,i-1)$ and ends at $v^{\prime}_i = (n-1,m-k+i)$ for each $i=1,2,\dots,k-1$.
    \end{itemize}

    To count the number of such families of non-intersecting lattice paths, we apply again the Lindstr\"om--Gessel--Viennot theorem (as stated in the proof of Theorem \ref{thm3}). This time, the $(i,j)$-entry of the matrix $B=[b_{i,j}]$ is given by the number of lattice paths going from $u^{\prime}_{i}$ to $v^{\prime}_{j}$ on the Young diagram of the shape obtained from $\bar{\lambda}/\bar{\mu}$ by deleting the first $k-1-j$ parts and the last $i-1$ parts. Using the notation \eqref{eq.gamma}, this shape can be written as $\gamma_{k-1-j,i-1}(\bar{\lambda})/\gamma_{k-1-j,i-1}(\bar{\mu})$. 

    Finally, by Lemma \ref{lem.skewshapepath}, the entry $b_{i,j}$ is given by
    \begin{equation}
        b_{i,j} = f(\gamma_{k-1-j,i-1}(\bar{\lambda}),\gamma_{k-1-j,i-1}(\bar{\mu})).
    \end{equation}
    We then obtain the desired formula \eqref{eq.thm4} by taking the determinant of the matrix $B$. This completes the proof of Theorem \ref{thm4}.   
\end{proof}

\section{Concluding remarks}\label{sec.remarks}

In this paper, we present a new idea to study maximal $I_k$-avoiding $(0,1)$-fillings of skew shapes $\lambda/\mu$ by viewing them as non-intersecting lattice paths going in the southwestern-northeastern direction on the square lattice. Our idea also works for maximal $J_k$-avoiding $(0,1)$-fillings of $\lambda/\mu$ when $\lambda/\mu$ is a rectangle. Due to symmetry, one can obtain maximal $J_k$-avoiding $(0,1)$-fillings of a rectangle by simply flipping a rectangle across its horizontal (or vertical) symmetry axis from those that are $I_k$-avoiding, in which case the corresponding non-intersecting lattice paths go in the northwestern-southeastern direction. It would be interesting to extend this framework to maximal $J_k$-avoiding $(0,1)$-fillings of general skew shapes and find more connections between $I_k$-avoiding and $J_k$-avoiding $(0,1)$-fillings. 

On the other hand, it is natural to consider $(0,1)$-matrices that avoid patterns other than $I_k$ or $J_k$, or that avoid multiple patterns simultaneously, and then analyze the structure or find the number of these matrices with maximally many $1$'s. We hope that our results will inspire further research on maximal $Q$-avoiding $(0,1)$-fillings of various shapes for a suitable class of patterns $Q$. 

\noindent \textbf{Acknowledgements.} 

The authors are grateful to Shen-Fu Tsai for pointing out an error in the earlier version of this paper. We also thank the anonymous reviewer for carefully reading the manuscript and providing helpful suggestions.

This research was supported by National Science and Technology Council, Taiwan, through grants 113-2115-M-003-010-MY3 (S.-P. Eu) and 114-2811-M-003-024 (Y.-L. Lee).




\end{document}